\documentclass[pdflatex,sn-aps]{sn-jnl}

\usepackage{graphicx}%
\usepackage{multirow}%
\usepackage{amsmath,amssymb,amsfonts}%
\usepackage{amsthm}%
\usepackage{mathrsfs}%
\usepackage[title]{appendix}%
\usepackage{xcolor}%
\usepackage{textcomp}%
\usepackage{manyfoot}%
\usepackage{booktabs}%
\usepackage{algorithm}%
\usepackage{algorithmicx}%
\usepackage{algpseudocode}%
\usepackage{listings}%
\usepackage{subfigure}
\theoremstyle{thmstyleone}%
\newtheorem{theorem}{Theorem}[section]
\newtheorem{lemma}{Lemma}[section]
\newtheorem{corollary}{Corollary}[section]

\theoremstyle{thmstyletwo}%
\newtheorem{remark}{Remark}%

\theoremstyle{thmstylethree}%

\begin{document}

\title[Article Title]{Scaling Optimized Spectral Approximations on Unbounded Domains: The Generalized Hermite and Laguerre Methods}


\author[1,2]{\fnm{Hao} \sur{Hu}}\email{huhao@lsec.cc.ac.cn}

\author*[2,1]{\fnm{HaiJun} \sur{Yu}}\email{hyu@lsec.cc.ac.cn}


\affil[1]{\orgdiv{School of Mathematical Sciences}, \orgname{University of Chinese Academy of Sciences}, \orgaddress{\street{Street}, \city{Beijing}, \postcode{100049}, \state{State}, \country{China}}}

\affil[2]{\orgdiv{State Key Laboratory of Mathematical Sciences (SKLMS) and State Key Laboratory of Scientific and Engineering Computing (LSEC)}, \orgdiv{Institute of Computational Mathematics and Scientific/Engineering Computing}, \orgname{Academy of Mathematics and Systems Science}, \orgaddress{\city{Beijing}, \postcode{100190}, \country{China}}}



\abstract{
We propose a novel error analysis framework for scaled generalized Laguerre and generalized Hermite approximations.
This framework can be regarded as an analogue of the Nyquist-Shannon sampling theorem: 
It characterizes the spatial and frequency bandwidths that can be 
effectively captured by Laguerre or Hermite sampling points. Provided a 
function satisfies the corresponding bandwidth constraints, it can be 
accurately approximated within this framework. The proposed framework is 
notably more powerful than classical theory --- it not only provides systematic guidance
for choosing the optimal scaling factor, but also predicts root-exponential and other intricate convergence behaviors that classical approaches fail to capture. Leveraging this framework, we conducted a detailed comparative study of Hermite and Laguerre approximations.
We find that functions with similar decay and oscillation characteristics may nonetheless display markedly different convergence rates. 
Furthermore, approximations based on two concatenated sets of Laguerre functions may offer significant advantages over those using a single set of Hermite functions.
}





\maketitle

\section{Introduction}
Spectral methods, based on orthogonal polynomials or functions as basis sets, are among the most widely used approaches for the numerical solution of partial differential equations(PDEs). Renowned for their high approximation accuracy, they play a pivotal role in scientific computing and engineering simulations.

On unbounded domains, Laguerre and Hermite polynomials are particularly well-suited. However, applying Laguerre (respectively, Hermite) polynomials to solve differential equations defined on the half-line (respectively, the whole real line) presents several challenges~\cite{gottlieb_numerical_1977,funaro1989computational,shen2000stable}:
1) When computing Laguerre or Hermite polynomials of high degree in floating-point arithmetic, one encounters overflow issues and instability due to round-off errors. 
2) Direct use of Laguerre and Hermite polynomials yields accurate approximations only over relatively small intervals, as the error estimates are typically derived in weighted Sobolev spaces equipped with rapidly decaying weights. 
3) Standard Laguerre and Hermite polynomials exhibit inferior resolution properties compared to other families of polynomials that are orthogonal on bounded domains.

Significant progresses have been made in addressing these issues. To tackle the first challenge, Funaro introduced a scaled Laguerre function~\cite{funaro1989computational}. Further advancements were recently proposed by Huang and Yu~\cite{huang2024improved}. For the second challenge, Shen advocated the use of Laguerre functions ---rather than polynomials--- as basis functions to form the spectral methods~\cite{shen2000stable}, an approach that also partially alleviates the first challenge. Regarding the third issue, it has been shown that the approximation quality of both Hermite and Laguerre polynomials/functions can be significantly improved by choosing appropriate scaling factors~\cite{tang1993hermite,shen2000stable}.
Thanks to these breakthroughs, the stability and efficiency of Laguerre and Hermite spectral methods have been substantially improved, leading to their widespread adoption across diverse scientific and engineering applications.

For instance, stable Laguerre methods have been developed for solving nonlinear PDEs on semi-infinite intervals~\cite{guo2000laguerre}.
Laguerre polynomials and functions have also been employed to construct composite spectral or spectral element methods for problems in unbounded domains~\cite{wang2008composite,guo2001composite,xu2002mixed,shen2006laguerre,azaiez2009laguerre,chen2012new,shen2016efficient}.
Additionally, Laguerre methods have been used to design spectral numerical integrators for ordinary differential equations~\cite{ben2007numerical,guo2008integration}, efficient time-splitting Laguerre--Hermite spectral methods for Bose--Einstein condensates~\cite{bao_fourth-order_2005, bao2008generalized}, and methods for optimal control
~\cite{masoumnezhad2020laguerre} and fractional differential equations~\cite{bhrawy2014efficient,chen2018laguerre}.

Similarly, Hermite spectral methods have found broad applicability across numerous fields. These include Schr\"odinger equations on unbounded domains~\cite{bao_fourth-order_2005,shen_error_2013,sheng_nontensorial_2021}, Vlasov and kinetic equations~\cite{shan_discretization_1998, mieussens_discrete-velocity_2000, gibelli_spectral_2006, cai_numerical_2010,Mizerova.She2018,funaro_stability_2021,zhang_error_2023}, 
high-dimensional PDEs and stochastic differential equations~\cite{luo_hermite_2013, zhang_sparse-grid_2013},
fluid dynamics and uncertainty quantification~\cite{cameron_orthogonal_1947,
meecham_wienerhermite_1964, orszag_dynamical_1967, xiu_modeling_2003,
venturi_wick-malliavin_2013, tang_discrete_2014, nobile_adaptive_2016,
wan_numerical_2019},
and even electronic design~\cite{manfredi_perturbative_2020, chen_analytical_2021}, among others.

Since standard Laguerre polynomials exhibit relatively poor resolution properties compared to other orthogonal polynomial families---a characteristic shared by Hermite polynomials~\cite{gottlieb_numerical_1977}---it has long been suggested that appropriate coordinate scaling can significantly
enhance the performance of both Laguerre- and Hermite-based spectral methods~\cite{tang1993hermite, shen2000stable}.
In practice, numerous algorithms incorporating scaling have been developed.
For instance, Ma et al.~\cite{ma_hermite_2005} proposed a time-dependent scaling for parabolic PDEs, while Xia et al.~\cite{xia_efficient_2021} introduced a frequency-dependent scaling strategy.

Nevertheless, a systematic framework for selecting the optimal scaling factor and explicitly revealing how the approximation error depends on this factor is still lacking.  At its core, this difficulty stems from fundamental limitations in the classical theories of approximation errors, which are typically formulated in finite-smoothness function spaces. Not only do these theories fail to characterize the approximation capabilities of scaled Laguerre functions, but they also cannot predict the root-exponential convergence rates observed when approximating analytic functions~\cite{wang2024convergence}, let alone other more intricate convergence behaviors.

Hermite approximations face analogous challenges in selecting optimal scaling factors and convergence characterization—issues that were systematically addressed in our earlier work~\cite{hu2024scaling}. 
The Laguerre case considered in this paper presents greater analytical difficulties, and our results here substantially generalize the Hermite findings of~\cite{hu2024scaling}.

Specifically, we propose a novel error analysis framework for scaled Laguerre approximation. 
This framework can be viewed as an analogue of the Nyquist-Shannon sampling theorem. 
We precisely characterize the spatial (physical) and frequency bandwidth that can be 
captured by the Laguerre sampling points. As long as a target function satisfies the corresponding bandwidth constraints, it admits a high-fidelity approximation within this framework. Our characterization surpasses classical theories in several key aspects, as demonstrated below.

First, our framework provides systematic guidance for scaling. By applying the optimal scaling factor, the convergence rate of generalized Hermite or Laguerre approximations can be markedly improved for a broad class of functions.
As shown in Section~\ref{sec:experiments}, functions that originally exhibit subgeometric convergence under standard Hermite/Laguerre approximation can be accelerated to geometric convergence through optimal scaling. Similarly, for functions converging algebraically, optimal scaling doubles the convergence order. Moreover, the well-documented inefficiency of Hermite/Laguerre Gaussian quadrature compared to Gaussian quadrature over bounded domains
~\cite{trefethen2022exactness,kazashi_suboptimality_2023} can be effectively mitigated via appropriate scaling (see Section~\ref{subsec:optimality}). Notably, Kazashi et al.~\cite{kazashi_suboptimality_2023} demonstrate that even with alternative quadrature weights, Gauss--Hermite nodes without scaling remain inherently suboptimal --- an observation that underscores the critical role of scaling in enhancing approximation performance.

Second, our framework successfully predicts root-exponential and other complex convergence behaviors that lie beyond the reach of classical theory. While the root-exponential convergence of Hermite and Laguerre approximations for analytic functions has only recently been rigorously established~\cite{wang2024convergence,wang2025convergence},
our analysis reveals its origin: it arises directly from the exponential decay rate of the functions' Fourier transform. Furthermore, our approach enables the derivation of non-root-exponential convergence rates (see Section~\ref{subsec: exp conv}), for which rigorous theoretical proofs are still lacking in the literature. 
Intriguingly, Shen et al.~\cite{shen_new_2000,shen_recent_2009,shen2011spectral} reported a root-exponential convergence rate when approximating algebraically decaying functions in the pre-asymptotic regime --- a phenomenon long considered puzzling, since classical theory predicts only algebraic convergence. Our framework offers a clear and complete explanation for this behavior (see Section~\ref{subsec: puzzling convergence}).

Finally, because our framework provides a more detailed characterization than classical theory, it enables a detailed comparative analysis of generalized Hermite and Laguerre approximations. We demonstrate that functions with similar decay and oscillatory characteristics may exhibit markedly different convergence behaviors depending on whether they are approximated using Hermite or Laguerre bases (see Section~\ref{subsec: vs}). We further compare the efficiency of a single set of Hermite functions against two concatenated sets of Laguerre functions. For a class of analytic functions, we show that the Laguerre-based approach performs at least as well as ---and often better than--- the Hermite counterpart (see Section~\ref{subsec:1 Hermite vs 2 Laguerre}).
We conjecture that this superior performance of dual Laguerre sets over a single Hermite set reflects a general principle, meriting more in-depth investigation in further work.

The remainder of this paper is organized as follows. Section~\ref{sec:basic} introduces essential preliminaries on generalized Laguerre and generalized Hermite functions used throughout the paper. Section~\ref{sec:intuitive}  offers an
intuitive overview of our main insights. The core theoretical result ---a novel estimate of the projection error--- is presented in Section~\ref{sec:main}.
Section~\ref{sec:general} extends these results to derivatives, interpolation, quadrature, and model PDEs. Numerical experiments and discussions appear in Section~\ref{sec:experiments}, followed by concluding remarks in Section~\ref{sec:conclusion}.

\section{Preliminaries}
\label{sec:basic}

\subsection{Generalized Laguerre polynomials and functions}
The generalized Laguerre polynomials (GLPs), denoted by $\mathscr{L}_n^{(\alpha)}(x)$,
(with $\alpha > -1$), are orthogonal polynomials with respect to the weight function
$\omega_\alpha^{+}(x)=x^\alpha e^{-x}$ on the half line $\mathbb{R}_{+}:=(0,+\infty)$,
i.e.,
\begin{equation}\label{eq:Laguerre-poly}
  \int_0^{+\infty} \mathscr{L}_n^{(\alpha)}(x) \mathscr{L}_m^{(\alpha)}(x) \omega_\alpha^{+}(x) d x=\gamma_n^{(\alpha)} \delta_{m n},
\end{equation}
where
\begin{equation}\label{eq:gamma-n}
  \gamma_n^{(\alpha)}=\frac{\Gamma(n+\alpha+1)}{n!}.
\end{equation}
The generalized Laguerre functions (GLFs) are defined by
\begin{equation}\label{eq:Laguerre-func}
  \widehat{\mathscr{L}}_n^{(\alpha)}(x):=e^{-x / 2} \mathscr{L}_n^{(\alpha)}(x), \quad x \in \mathbb{R}_{+}, \quad \alpha>-1 .
\end{equation}
From \eqref{eq:Laguerre-poly} and \eqref{eq:gamma-n}, the GLFs are orthogonal with respect to the weight function
$\hat{\omega}_\alpha^{+} = x^\alpha$, i.e.,
\begin{equation}\label{eq:Laguerre-func-orth}
  \int_0^{+\infty} \widehat{\mathscr{L}}_n^{(\alpha)}(x) \widehat{\mathscr{L}}_m^{(\alpha)}(x) \hat{\omega}_\alpha^{+}(x) d x=\gamma_n^{(\alpha)} \delta_{m n}.
\end{equation}

\subsection{Generalized Hermite polynomials and functions}
The generalized Hermite polynomials (GHPs), denoted by $H_n^{(\mu)}(x)$
(with $\mu > -\frac{1}{2}$), are orthogonal polynomials with respect to the weight function
$\omega_{\mu}(x)=|x|^{2\mu} e^{-x^2}$ on the whole line $\mathbb{R}:=(-\infty,+\infty)$,
i.e.,
\begin{equation}\label{eq:Hermite-poly}
  \int_{\mathbb{R}} H_m^{(\mu)}(x) H_n^{(\mu)}(x)|x|^{2 \mu} \mathrm{e}^{-x^2} \mathrm{~d} x=\gamma_{n,H}^{(\mu)} \delta_{m n},
\end{equation}
where
\begin{equation}\label{eq:gamma-n-H}
  \gamma_{n,H}^{(\mu)}=2^{2 n}\left[\frac{n}{2}\right]!\Gamma\left(\Bigl[\frac{n+1}{2}\Bigr]+\mu+\frac{1}{2}\right).
\end{equation}
For $\mu > -\frac{1}{2}$, we define the generalized Hermite function (GHF) of degree $n$ with parameter $\mu$ by
\begin{equation}\label{eq:Hermite-func}
  \widehat{H}_n^{(\mu)}(x):=\sqrt{1 / \gamma_{n,H}^{(\mu)}} \mathrm{e}^{-\frac{x^2}{2}} H_n^{(\mu)}(x), \quad n \geqslant 0, \quad x \in \mathbb{R}.
\end{equation}
Then, we have the orthogonality:
\begin{equation}\label{eq:Hermite-func-orth}
  \int_{\mathbb{R}} \widehat{H}_l^{(\mu)}(x) \widehat{H}_n^{(\mu)}(x)|x|^{2 \mu} \mathrm{~d} x=\delta_{l n}.
\end{equation}
It can be verified that the GHFs satisfy the following three-term
recurrence relation
\begin{equation}\label{eq:0823-3}
  \begin{aligned}
  \widehat{H}_0^{(\mu)}(x) &=\sqrt{1 / \Gamma(\mu+1 / 2)} \mathrm{e}^{-\frac{x^2}{2}}, \\
  \quad \widehat{H}_1^{(\mu)}(x) &=\sqrt{1 / \Gamma(\mu+3 / 2)} x\,\mathrm{e}^{-\frac{x^2}{2}},\\
  \widehat{H}_{n+1}^{(\mu)}(x) &=a_n^{(\mu)} x \widehat{H}_n^{(\mu)}(x)-c_n^{(\mu)} \widehat{H}_{n-1}^{(\mu)}(x), \quad n \geqslant 1, \\
  \end{aligned}
\end{equation}
where
\begin{equation}\label{eq:0823-4}
  a_n^{(\mu)}=\sqrt{\frac{2}{n+1+2 \mu-\theta_n}}, \quad c_n^{(\mu)}=\frac{n+\theta_n}{\sqrt{\left(n+\theta_n /(2 \mu)\right)\left(n+2 \mu+1-\theta_n /(2 \mu)\right)}} ,
\end{equation}
$\theta_n$ is defined as
\begin{equation}\label{eq:0823-5}
  \theta_n= \begin{cases}0, & n \text { even }, \\ 2 \mu, & n \text { odd }.\end{cases}
\end{equation}
From \eqref{eq:0823-3}, we can further derive the recurrence formula
for the derivative of GHFs:
\begin{lemma}\label{lem:derivative recurrence}
When $2 \mid n$, 
\begin{equation}\label{eq:0823-6}
  \partial_x \widehat{H}_n^{(\mu)}=-\sqrt{\frac{n+1+2 \mu}{2}} \widehat{H}_{n+1}^{(\mu)}+\sqrt{\frac{n}{2}} \widehat{H}_{n-1}^{(\mu)},
\end{equation}
when $2 \nmid n$,
\begin{equation}\label{eq:0823-7}
  \begin{aligned}
  \partial_x \widehat{H}_n^{(\mu)} & =
  -\sqrt{\frac{n+1}{2}} \widehat{H}_{n+1}^{(\mu)}
  +\sqrt{\frac{n+2 \mu}{2}} \widehat{H}_{n-1}^{(\mu)}  \\
  &{}\quad + 2 \mu (-1)^{\frac{n-3}{2}} \sqrt{\frac{\left(\frac{n-1}{2}\right)!}{\Gamma\left(\frac{n}{2}+\mu+1\right)}}
   \sum_{k=0}^{\frac{n-1}{2}}(-1)^k \sqrt{\frac{\Gamma\left(k+\mu+\frac{1}{2}\right)}{k!}} \widehat{H}_{2 k}^{(\mu)}.
  \end{aligned}
\end{equation}
\end{lemma}

\subsection{Relations between Laguerre and Hermite polynomials}
The GHPs can be expressed in terms of GLPs:
\begin{equation}\label{eq:relaion-Laguerre-Hermite}
  H_n^{(\mu)}(x)= \begin{cases}(-1)^{\frac{n}{2}} 2^n\left(\frac{n}{2}\right)!\mathscr{L}_{n / 2}^{(\mu-1 / 2)}\left(x^2\right), & n \text { even, } \\ (-1)^{\frac{n-1}{2}} 2^n\left(\frac{n-1}{2}\right)!x \mathscr{L}_{(n-1) / 2}^{(\mu+1 / 2)}\left(x^2\right), & n \text { odd, }\end{cases}
\end{equation}

\subsection{Some useful Fourier transform results}
Let $\mathcal{F}[f(u)](x)$ denotes the Fourier transform of $f(u)$, which
is defined as
\begin{equation}\label{eq:0823-1}
  \mathcal{F}[f(u)](x)=\frac{1}{\sqrt{2 \pi}} \int_{-\infty}^{\infty} e^{-i x u} f(u) d u.
\end{equation}

In the following lemma, we present some useful Fourier transform results that will be used
to construct the main result, Theorem~\ref{thm:bigthm}.
These results can be found in (14.9), (14.10) of~\cite{chihara1955generalized}.
\begin{lemma}\label{lem:transform}
  \begin{equation}\label{eq:0823-2}
    \begin{aligned}
    & \sqrt{2}\mathcal{F}\left[e^{-u^2}|u|^{2 \mu} H_{2 m}^\mu(u)\right](-2x)= \\
    & \qquad\qquad(-1)^m 2^{2 m} \frac{\Gamma\left(m+\mu+\frac{1}{2}\right)}{\Gamma\left(m+\frac{1}{2}\right)} x^{2 m} e^{-x^2} \Phi\left(-\mu, m+\frac{1}{2} ; x^2\right), \\
    & \sqrt{2}\mathcal{F}\left[e^{-u^2}|u|^{2 \mu} H_{2 m+1}^\mu(u)\right](-2x)= \\
    & \qquad \qquad(-1)^m 2^{2 m+1} \frac{\Gamma(m+\mu+3 / 2)}{\Gamma(m+3 / 2)} i\, x^{2 m+1} e^{-x^2} \Phi\left(-\mu, m+3 / 2 ; x^2\right) .
    \end{aligned}
  \end{equation}
\end{lemma}

\section{A quick and intuitive understanding of our results}
\label{sec:intuitive}
The traditional approach to studying convergence in Laguerre/Hermite spectral methods focuses on examining the eigenvalues that arise from the Sturm--Liouville problems corresponding to these polynomial families. 
As an example, the generalized Laguerre polynomial is a solution of the Sturm--Liouville differential equation
\begin{equation}
  x^{-\alpha} e^x \partial_x\left(x^{\alpha+1} e^{-x} \partial_x \mathscr{L}_n^{(\alpha)}(x)\right)+\lambda_n \mathscr{L}_n^{(\alpha)}(x)=0,
\end{equation}
with the corresponding eigenvalue $\lambda_n = n$. As emphasized in~\cite{shen2011spectral}, $\lambda_n$ increases linearly rather than quadratically as in the Jacobi case. Consequently, the convergence rate of expansions using GLPs is only half that of expansions based on Jacobi polynomials with ``comparable'' regularity.

Nevertheless, this strategy cannot be effectively employed to
analyze the approximation properties of scaled Laguerre (Hermite)
polynomials $\mathscr{L}_n^{(\alpha)}(\beta x)$ ($\widehat{H}_n^{(\mu)}(\beta x)$). 
Estimates of this type are widely available in the literature.
For example, Theorem 4.1 in~\cite{liu2017fully} asserts:
\begin{equation}\label{eq:liu2017}
  \bigl\|u-\pi_N^{\alpha, \beta} u\bigr\|_{s, \alpha, \beta} \lesssim(\beta N)^{\frac{s-r}{2}}\bigl\|u-\pi_N^{\alpha, \beta} u\bigr\|_{r, \alpha, \beta} \lesssim(\beta N)^{\frac{s-r}{2}}\|u\|_{r, \alpha, \beta}, \quad r \geq s.
\end{equation}
This result yields an estimate for the projection error when approximating 
$u$ by the first $N+1$ scaled Laguerre polynomials $\mathscr{L}_n^{(\alpha)}(\beta x)$.
While the right-hand side of \eqref{eq:liu2017} appears to indicate a convergence rate of $(\beta N)^{\frac{s-r}{2}}$, one must keep in mind that the norm itself also varies with the scaling parameter. Consequently, it becomes challenging to disentangle and assess the precise impact of the scaling factor $\beta$ on the approximation error.

In contrast, our method follows a fundamentally different line of reasoning, based on an observation about the root distribution of Hermite polynomials. Take $H_{2N+2}^{(\alpha + \frac{1}{2})}(\beta x)$ as an example. Let $\{x_j^{(\alpha)}\}_{j=0}^N$ denote its positive roots arranged in ascending order, then from \eqref{eq:relaion-Laguerre-Hermite} we know that $ \bigl(\beta\, x_j^{(\alpha)}\bigr)^2,\, 0\leqslant j \leqslant N$
are roots of $\mathscr{L}_{N+1}^{(\alpha)}(x)$. Hence, by Theorem 8.9.2 of~\cite{szego1975orthogonal},
\begin{equation}\label{eq:shen2011-7.42}
  x_j^{(\alpha)}=\frac{(j+1) \pi+O(1)}{2\beta \sqrt{N+1}}, \text { for } x_j^{(\alpha)} \in(0, \eta/\beta],\; \alpha>-1,
\end{equation}
where $\eta > 0$ is a fixed constant. Moreover, by Theorem 6.31.3 of~\cite{szego1975orthogonal},
\begin{equation}\label{eq:shen2011-7.43}
  \left(x_j^{(\alpha)}\right)^2<(j+(\alpha+3) / 2) \frac{2 j+\alpha+3+\sqrt{(2 j+\alpha+3)^2+1 / 4-\alpha^2}}{\beta^2 \left(N+(\alpha+3) / 2\right)},
\end{equation}
for all $0 \leqslant j \leqslant N$ and $\alpha > -1$. In particular,
\begin{equation}\label{eq:shen2011-7.44}
  x_N^{(\alpha)}=\frac{1}{\beta}\sqrt{4 N+2 \alpha+6+O\left(N^{1 / 3}\right)}.
\end{equation}
Overall, the zeros of $H_N^{(\mu)}(\beta x)$ lie within the interval
$\left[-\mathcal{O}\left(\sqrt{N}/\beta\right), \mathcal{O}\left(\sqrt{N}/\beta\right)\right]$, and the smallest distance between neighboring zeros is typically on the order of $\frac{1}{\beta \sqrt{N}}$.

As already observed in \cite{tang1993hermite}, because the zeros of $H_N^{(\mu)}(\beta x)$—and thus the sampling nodes for Hermite interpolation—are confined to an interval of size ${\sqrt{N}}/{\beta}$, a good approximation based on the first $N+1$ scaled Hermite polynomials $H_n^{(\mu)}(\beta x)$ can only be expected if the target function has decayed sufficiently outside this interval. We refer to ${\sqrt{N}}/{\beta}$ as the spatial bandwidth, and we call the error caused by insufficient decay of the target function outside this spatial bandwidth the \emph{spatial truncation error}.

Our analysis begins with the additional observation that, because the spacing between neighboring interpolation points is approximately $\frac{1}{\beta \sqrt{N}}$, any components of the objective function with frequencies exceeding $\beta \sqrt{N}$ cannot be accurately approximated. This situation is comparable to aliasing in the Shannon sampling theorem, where high-frequency components are inaccurately captured. We therefore refer to $\beta\sqrt{N}$ as the frequency bandwidth, and we define the error caused by the insufficient decay of the Fourier transform of the objective function outside this bandwidth as the \emph{frequency truncation error}.

We conclude that the error in generalized Hermite approximation can be controlled by the spatial truncation error, the frequency truncation error, and an exponentially decaying spectral error (see Theorem \ref{thm:bigthm}). The case of Laguerre approximation can be transformed and analyzed in terms of generalized Hermite approximation.

\section{Main result: projection error estimate}
\label{sec:main}
We begin by defining the generalized Laguerre and Hermite projections. We then establish the equivalence between approximations based on these two orthogonal systems. Finally, we present an error estimation theorem for scaled Hermite approximation. Owing to this equivalence, the corresponding error bounds for Laguerre approximation follow immediately via a straightforward transformation.

\subsection{Definition of projections}
Let $P_N$ denote the set of polynomials of degree at most $N$. We define $\widehat{P}_N^{\beta,L}$ as
\begin{equation}\label{eq:def-PNHat-L}
  \widehat{P}_N^{\beta,L}:=\bigl\{\phi: \phi=e^{- \beta x / 2} \psi, \quad \forall\, \psi \in P_N\bigr\} .
\end{equation}
We define projection operator $\widehat{\Pi}_{N}^{(\alpha,\beta,L)}: L^2_{\hat{\omega}_\alpha^+}(\mathbb{R}_+) \rightarrow \widehat{P}_N^{\beta,L}$ by
\begin{equation}\label{eq:def-projHat-L}
  \left(u-\widehat{\Pi}_{N}^{(\alpha,\beta,L)} u, v_N\right)_{x^\alpha}=0, \quad \forall\, v_N \in \widehat{P}_N^{\beta,L}.
\end{equation}
Then, we have
\begin{equation}\label{eq:0709-1}
  \widehat{\Pi}_{N}^{(\alpha,\beta,L)} u=\sum_{n=0}^N c_n \widehat{\mathscr{L}}_n^{(\alpha)}(\beta x), \ c_n=\frac{1}{\bigl\|\widehat{\mathscr{L}}_n^{(\alpha)}(\beta x)\bigr\|_{x^\alpha}^2} \int_{\mathbb{R}_{+}} u(x) \widehat{\mathscr{L}}_n^{(\alpha)}(\beta x) x^\alpha d x.
\end{equation}
That is to say, $\widehat{\Pi}_{N}^{(\alpha,\beta,L)} u$ is the sum of the first
$N+1$ terms of the Laguerre expansion.

Similarly, we define $\widehat{P}_N^{\beta,H}$ as
\begin{equation}\label{eq:def-PNHat-H}
  \widehat{P}_N^{\beta,H}:=\left\{\phi: \phi=e^{- \beta^2 x^2 / 2} \psi, \quad \forall\, \psi \in P_N\right\} .
\end{equation}
Let $\omega_\mu = |x|^{2\mu}$, we can define 
projection operator $\widehat{\Pi}_{N}^{(\mu,\beta)}: L^2_{\omega_\mu}(\mathbb{R}) \rightarrow \widehat{P}_N^{\beta,H}$ by
\begin{equation}\label{eq:def-projHat-H}
  \left(u-\widehat{\Pi}_{N}^{(\mu,\beta)} u, v_N\right)_{\omega_\mu}=0, \quad \forall\, v_N \in \widehat{P}_N^{\beta,H}.
\end{equation}
So, we have
\begin{equation}\label{eq:0709-2}
  \widehat{\Pi}_{N}^{(\mu,\beta)} u=\sum_{n=0}^N c_n \widehat{H}_n^{(\mu)}(\beta x), \ 
  c_n=\frac{1}{\bigl\|\widehat{H}_n^{(\mu)}(\beta x)\bigr\|_{|x|^{2\mu}}^2} \int_{\mathbb{R}} u(x) \widehat{H}_n^{(\mu)}(\beta x) |x|^{2\mu} d x.
\end{equation}
That is to say, $\widehat{\Pi}_{N}^{(\mu,\beta)} u$ is the sum of the first
$N+1$ terms of the Hermite expansion.

By simply changing the variables, we can prove the following lemma:
\begin{lemma}\label{lem:equiv}
  If $v(x) \in L^2_{x^{\alpha}}(\mathbb{R}_+)$, define $u(x) \in L^2_{|x|^{2\mu}}(\mathbb{R})$ by $u(x)=v(x^2)$, where $\alpha=\mu-1/2$, then
  \begin{equation}\label{eq:equiv}
    \left\|u-\widehat{\Pi}_{2N}^{(\mu,\beta)} u\right\|_{|x|^{2 \mu}}
    =\left\|v-\widehat{\Pi}_{N}^{(\alpha,\beta^2,L)} v \right\|_{x^{\alpha}}.
  \end{equation}
\end{lemma}
From \eqref{eq:0709-1}, \eqref{eq:0709-2} and Lemma \ref{lem:equiv} we know that the
generalized Laguerre approximation is equivalent to the generalized Hermite
approximation. Hence, we can understand the behavior of the Laguerre approximation by
equivalently, considering the Hermite case. 

\subsection{Projection error estimate}
In the next theorem, we give an error estimate when using $N+1$
truncated terms of scaled generalized Hermite functions 
$\widehat{H}^{\left(\mu\right)}_n(\beta x)$ to approximate a function $u$.
This error estimate has three different components: spatial truncation
error, frequency truncation error, and spectral error.

\begin{theorem}\label{thm:bigthm}
  If $u \cdot\bigl(1+x^2\bigr)^{{\mu}/{2}} \in L^2$,
  where $0 \leqslant \mu \in \mathbb{R}$, let $\widehat{\Pi}_N^\beta u$
  denote $\widehat{\Pi}_{N}^{(\mu,\beta)}$ defined in \eqref{eq:def-projHat-H}, and
  $\omega = |x|^{2\mu}$. Let $M = \frac{\sqrt{N}}{2\sqrt{3}}/\beta,\, B = \frac{\sqrt{N}}{2\sqrt{3}} \beta$  
  being the spatial and frequency bandwidths, respectively, that can be resolved when approximating
  $u$ by $\widehat{\Pi}_N^\beta u$. Then we have
  \begin{equation}\label{eq:bigthm}
    \begin{aligned}
    \left\|u-\widehat{\Pi}_N^\beta u\right\|_\omega & 
    \lesssim_\mu \bigl\|u \cdot \mathbb{I}_{\left\{|x|>M\right\}}\bigr\|_\omega 
     +B^{-\mu+\frac{1}{2}} \bigl\| \mathcal{F}[u]\left(B \xi\right) \bigr\|_{H^{\mu}\left(\mathbb{R}\backslash [-1,1]\right)} \\
    &{} \quad +e^{-\frac{N}{24}}\left(\|u\|_\omega+\|u\| \cdot M^\mu\right),
    \end{aligned}
  \end{equation}
  where $f \lesssim_\mu g$ means $f \leqslant C_{\mu} g$,  $C_{\mu}$ is a constant that depends only on $\mu$.
\end{theorem}

\begin{remark}
  This theorem can be regarded as an analogue of the Nyquist-Shannon sampling theorem: 
  When using the first $N+1$ scaled Hermite functions $\widehat{H}_n^{(\mu)}(\beta x)$, the capturable bandwidth ranges in the spatial and frequency domains are $M = \mathcal{O}\bigl(\sqrt{N}/\beta\bigr)$ and $B = \mathcal{O}\bigl(\sqrt{N}\beta\bigr)$, respectively. 
  The approximation error arises from three sources: 
  \begin{enumerate}
    \item $\left\|u \cdot \mathbb{I}_{\left\{|x|>M\right\}}\right\|_\omega$:
    The insufficient decay of components beyond the spatial bandwidth.
    \item $B^{-\mu+\frac{1}{2}} \left\| \mathcal{F}[u]\left(B \xi\right) \right\|_{H^{\mu}\left(\mathbb{R}\backslash [-1,1]\right)}$:
    the high-frequency components beyond the frequency bandwidth.
    \item $e^{-\frac{N}{24}}\left(\|u\|_\omega+\|u\| \cdot M^\mu\right)$:
    The exponentially decaying error that is attributed to the approximation 
    within the truncated bandwidth range.
  \end{enumerate}
  The following proof strategy is built upon this intuitive understanding.
\end{remark}

\begin{proof}
  Let $v(x) = u(x/\beta)$, then
  \begin{equation}\label{eq:0505-2}
    \left\|u-\widehat{\Pi}_N^\beta u\right\|_\omega=\beta^{-\mu-\frac{1}{2}}\left\|v-\widehat{\Pi}_N^1 v\right\|_\omega.
  \end{equation}
  Similarly, after replacing $u$ with $v$ and $\beta$ with $1$, 
  each term on the right-hand side of \eqref{eq:bigthm} 
  will also acquire the same additional factor $\beta^{-\mu-\frac{1}{2}}$,
  hence it is sufficient to prove \eqref{eq:bigthm} for $\beta = 1$.
  In this case, we have $M = B = \frac{\sqrt{N}}{2\sqrt{3}}$.

  Let us define the spatial truncation function as
  \begin{equation}\label{eq:0505-3}
    T_M^s=\mathbb{I}_{[-2 M, 2 M]} * \frac{1}{\sqrt{2 \pi}} e^{-\frac{1}{2} x^2}.
  \end{equation}
  Then define
  \begin{equation}\label{eq:0505-3.5}
    u_M(x) = u(x) \cdot T_M^s(x).
  \end{equation}
  Let
  \begin{equation}\label{eq:0505-4}
    \alpha(x)=\left\{\begin{array}{cc}
    e^{\frac{1}{x^2-1}} & |x|<1 \\
    0 & |x| \geqslant 1,
    \end{array}\right.
  \end{equation}
  and
  \begin{equation}\label{eq:0505-5}
    \gamma(x)=\left(\int_\mathbb{R} \alpha(x) d x\right)^{-1} \alpha(x) ; \quad \gamma_{\varepsilon}(x)=\varepsilon^{-1} \gamma\left(\frac{x}{\varepsilon}\right).
  \end{equation}
  Then, define the frequency truncation function
  \begin{equation}\label{eq:0505-6}
    T_B^f=\mathbb{I}_{\left[-\frac{3}{2} B, \frac{3}{2} B\right]} * \gamma_{\frac{B}{2}}.
  \end{equation}
  Further define
  \begin{equation}\label{eq:0505-7}
    u^B(x)=\frac{1}{\sqrt{2 \pi}} \int \mathcal{F}[u](\xi) \cdot T_B^f(\xi) \cdot e^{i \xi x} d \xi,  
  \end{equation}
  and
  \begin{equation}
  u_M^B=u^B \cdot T_M^S.
  \end{equation}
  With the above preparation, we can perform the following estimation. Since $\beta = 1$, we
  will ignore the superscript $\beta$ in $\widehat{\Pi}_N^\beta$ for simplicity.
  \begin{equation}\label{eq:0505-9}
    \begin{aligned}
    \left\|u-\widehat{\Pi}_N u\right\|_\omega & \leqslant\left\|u-u_M^B\right\|_\omega+\left\|u_M^B-\widehat{\Pi}_N u_M^B\right\|_\omega+\left\|\widehat{\Pi}_N\left(u_M^B-u\right)\right\|_\omega \\
    & \lesssim\left\|u-u_M^B\right\|_\omega+\left\|u_M^B-\widehat{\Pi}_N u_M^B\right\|_\omega \\
    & \lesssim\left\|u-u_M\right\|_\omega+\left\|u_M-u_M^B\right\|_\omega+\left\|u_M^B-\widehat{\Pi}_N u_M^B\right\|_\omega \\
    & \triangleq E_s+E_f+E_H.
    \end{aligned}
  \end{equation}
  Notice that
  \begin{equation}\label{eq:0505-10}
    \begin{aligned}
    E_s^2 & =\int\left(u-u_M\right)^2 \cdot \omega \\
    & =\int_{|x| \leqslant M}\left(u-u_M\right)^2 \cdot \omega+\int_{|x|>M}\left(u-u_M\right)^2 \cdot \omega.
    \end{aligned}
  \end{equation}
  For the first term,
  \begin{equation}\label{eq:0505-11}
    \int_{|x| \leqslant M}(u-u_M)^2 \cdot \omega=\int_{|x| \leqslant M} u^2\left(1-T_M^s\right)^2 \cdot \omega.
  \end{equation}
  Since for $|x| \leqslant M$
  \begin{equation}\label{eq:0505-12}
    \begin{aligned}
    \left|1-T_M^s\right| & =\left(\int_{2 M-x}^{+\infty}+\int_{-\infty}^{-2 M-x}\right) \frac{1}{\sqrt{2 \pi}} e^{-\frac{1}{2} t^2} d t \\
    & \lesssim \int_M^{+\infty} \frac{1}{\sqrt{2 \pi}} e^{-\frac{1}{2} t^2} d t \\
    & \lesssim e^{-\frac{1}{2} M^2},
    \end{aligned}
  \end{equation}
  putting \eqref{eq:0505-12} into \eqref{eq:0505-11} yields
  \begin{equation}\label{eq:0505-13}
    \begin{aligned}
    \int_{|x| \leqslant M}\left(u-u_M\right)^2 \cdot \omega & \lesssim e^{-M^2} \int u^2 \cdot \omega =e^{-M^2}\|u\|_\omega^2.
    \end{aligned}
  \end{equation}
  For the second term of \eqref{eq:0505-10} we have
  \begin{equation}\label{eq:0505-14}
    \int_{|x|>M}\left(u-u_M\right)^2 \cdot \omega=\int_{|x|>M} u^2\left(1-T_M^s\right)^2 \cdot \omega.
  \end{equation}
  Notice that $0 \leqslant T_M^s \leqslant 1$, we have
  \begin{equation}\label{eq:0505-15}
    \begin{aligned}
    \int_{|x|>M}\left(u-u_M\right)^2 \cdot \omega & \leqslant \int_{|x|>M} u^2 \cdot \omega 
     =\left\|u \cdot \mathbb{I}_{\{|x|>M\}}\right\|_\omega^2.
    \end{aligned}
  \end{equation}
  Putting \eqref{eq:0505-13}, \eqref{eq:0505-15} back into \eqref{eq:0505-10}
  yields
  \begin{equation}\label{eq:0505-16}
    \begin{aligned}
    E_s & \lesssim e^{-\frac{1}{2} M^2}\|u\|_\omega+\|u \cdot \mathbb{I}_{\{|x|>M\}}\|_\omega \\
    & =e^{-\frac{N}{24}}\|u\|_\omega+\left\|u \cdot \mathbb{I}_{\left\{|x|>M\right\}}\right\|_\omega.
    \end{aligned}
  \end{equation}
  For $E_f$ defined in \eqref{eq:0505-9} we have
  \begin{equation}\label{eq:0505-17}
    \begin{aligned}
    E_f^2 & =\int\left|u_M-u_M^B\right|^2 \cdot \omega \\
    & =\int\left|u-u^B\right|^2 \cdot\left(T_M^s\right)^2 \cdot \omega \\
    & \leqslant \int\left|u-u^B\right|^2 \cdot \omega.
    \end{aligned}
  \end{equation}
  The inequality in the above formula is due to $0 \leqslant T_M^s \leqslant 1$.
  Recall that the weight function $\omega = |x|^{2\mu}$, let
  \begin{equation}\label{eq:1111-0.5}
    y=B x, \quad d(y)=u(x)-u^B(x),
  \end{equation}
  then 
  \begin{equation}\label{eq:1111-1}
    \begin{aligned}
     \int\left|u-u^B\right|^2 \cdot \omega d x 
    &=  B^{-2 \mu-1} \int|d(y)|^2|y|^{2 \mu} d y \\
    &\leqslant  B^{-2 \mu-1} \int|d(y)|^2\bigl(1+y^2\bigr)^\mu d y.
    \end{aligned}
  \end{equation}
  Combining this with the definition of the $H^\mu$ norm, we have
  \begin{equation}\label{eq:1111-2}
    \int\left|u-u^B\right|^2 \cdot \omega d x \lesssim B^{-2 \mu-1}\bigl\|\mathcal{F}[d](\xi)\bigr\|_{H^\mu}^2.
  \end{equation}
  From \eqref{eq:0505-7} we have
  \begin{equation}\label{eq:0505-22}
    \mathcal{F}\bigl[u^B\bigr](k)=\mathcal{F}[u](k) \cdot T_B^f(k),
  \end{equation}
  then combining \eqref{eq:1111-0.5} yields
  \begin{equation}\label{eq:1111-3}
    \mathcal{F}[d](\xi)=B \cdot \mathcal{F}[u](B \xi) \cdot \left(1-T_1^f(\xi)\right).
  \end{equation}
  Putting \eqref{eq:1111-3} back into \eqref{eq:1111-2}, we have
  \begin{equation}\label{eq:1111-4}
    \int \left|u-u^B\right|^2 \cdot \omega d x \lesssim B^{-2 \mu+1} \left\| \mathcal{F}[u](B \xi)\left(1-T_1^f(\xi)\right)\right\|_{H^\mu}^2.
  \end{equation}
  Let $g(\xi) = \mathcal{F}[u](B\xi)$, we assert that the following conclusion holds:
  \begin{equation}\label{eq:1111-5}
    \left\|g(\xi)\left(1-T_1^f(\xi)\right)\right\|_{H^\mu(\mathbb{R})} \lesssim_\mu \bigl\|g(\xi)\bigr\|_{H^{\mu}\left(\mathbb{R}\backslash [-1,1]\right)}.
  \end{equation}
  By the definition of $T_B^f(\xi)$ \eqref{eq:0505-6} we have
  \begin{equation}\label{eq:0505-23}
    0 \leqslant T_1^f(\xi) \leqslant 1; \quad T_1^f(\xi) \equiv 1,\ \forall\,|\xi| \leqslant 1.
  \end{equation}
  Hence for $\mu \in \mathbb{N}$, we have
  \begin{equation}\label{eq:1111-6}
    \left\|g\left(1-T_1^f\right)\right\|_{H^\mu(\mathbb{R})}^2=\sum_{j=0}^\mu \int_{\mathbb{R} \backslash[-1,1]}\left|\partial^j\left(g\left(1-T_1^f\right)\right)\right|^2 d \xi.
  \end{equation}
  Combining
  \begin{equation}\label{eq:1111-7}
    \begin{aligned}
    \left|\partial^j\left(g\left(1-T_1^f\right)\right)\right|^2 & =\left|\sum_{r=0}^j\binom{j}{r}\left(\partial^r g\right) \cdot\left(\partial^{j-r}\left(1-T_1^f\right)\right)\right|^2 \\
    & \lesssim_\mu \sum_{r=0}^\mu\left|\partial^r g\right|^2
    \end{aligned}
  \end{equation}
  yields
  \begin{equation}\label{eq:1111-8}
    \left\|g\left(1-T_1^f\right)\right\|_{H^\mu(\mathbb{R})}^2 \lesssim_\mu \sum_{j=0}^\mu \int_{\mathbb{R} \backslash[-1,1]}\bigl|\partial^j g\bigr|^2 d \xi.
  \end{equation}
  That is to say, \eqref{eq:1111-5} holds for $\mu \in \mathbb{N}$.
  It then follows from the interpolation theory of Sobolev spaces that \eqref{eq:1111-5} holds for any $\mu \geqslant 0$.
  Substituting \eqref{eq:1111-5} into \eqref{eq:1111-4} then into \eqref{eq:0505-17}, we obtain
  \begin{equation}\label{eq:1111-9}
    E_f \lesssim_\mu B^{-\mu+\frac{1}{2}} \bigl\|\mathcal{F}[u](B \xi)\bigr\|_{H^{\mu}\left(\mathbb{R}\backslash [-1,1]\right)}.
  \end{equation}
  Next we consider the estimation of $E_H$ defined in \eqref{eq:0505-9}.
  Let
  \begin{equation}\label{eq:0505-33}
    g_{\xi, s}(x)=e^{-\frac{1}{2}(x-s)^2+i \xi x},
  \end{equation}
  then
  \begin{equation}\label{eq:0505-34}
    \begin{aligned}
    u_M^B & =u^B \cdot T_M^s \\
    & =\frac{1}{2 \pi} \int_{|\xi| \leqslant 2 B,|s| \leqslant 2 m} \mathcal{F}\bigl[u^B\bigr](\xi) g_{\xi, s}(x) d \xi d s.
    \end{aligned}
  \end{equation}
  From \ref{lem:transform} we know:
  let $z = \xi - is$, then the generalized Hermite expansion of $g_{\xi,s}$
  \begin{equation}\label{eq:0505-35}
    g_{\xi, s}(x)=\sum_n c_n \widehat{H}_n^{(\mu)}
  \end{equation}
  satisfy: when $n = 2m,\, m \in \mathbb{N}$,
  \begin{equation}\label{eq:0505-36}
    \begin{aligned}
    c_{2 m} & =\sqrt{\pi} \times(-1)^m \times \sqrt{\frac{\Gamma\left(m+\mu+\frac{1}{2}\right)}{m!}} \times \frac{1}{\Gamma\left(m+\frac{1}{2}\right)} \times\left(\frac{z}{2}\right)^{2 m} \\
    &{}\qquad \times e^{-\frac{z^2}{4}-\frac{1}{2} s^2} \times \phi\left(-\mu, m+\frac{1}{2} ; \frac{z^2}{4}\right),
    \end{aligned}
  \end{equation}
  when $n = 2m+1,\, m \in \mathbb{N}$,
  \begin{equation}\label{eq:0505-37}
    \begin{aligned}
    c_{2 m+1} & =\sqrt{\pi} \times(-1)^m \times \sqrt{\frac{\Gamma\left(m+\mu+\frac{3}{2}\right)}{m!}} \times \frac{1}{\Gamma\left(m+\frac{3}{2}\right)} \times i \times\left(\frac{z}{2}\right)^{2 m+1} \\
    &{}\qquad \times e^{-\frac{z^2}{4}-\frac{1}{2} s^2} \times \phi\left(-\mu, m+\frac{3}{2} ; \frac{z^2}{4}\right).
    \end{aligned}
  \end{equation}
  Here, $\phi\left(a,c;x\right)$ in \eqref{eq:0505-36} and \eqref{eq:0505-37} is the so-called
  Kummer's confluent hypergeometric function, which is defined as
  \begin{equation}\label{eq:0505-38}
    \phi(a, c ; x)=1+\frac{a}{c} \frac{x}{1!}+\frac{a(a+1)}{c(c+1)} \frac{x^2}{2!}+\ldots
  \end{equation}
  Since $z = \xi - is$, when $|\xi|,|s| \leqslant 2 B=2 M=\sqrt{\frac{N}{3}}$, we have
  \begin{displaymath}
    \left|\frac{z^2}{4}\right| \leqslant \frac{N}{6}.
  \end{displaymath}
  Then for $n > N,\, m=\left[\frac{n}{2}\right]$ we have
  \begin{equation}\label{eq:0505-39}
    \left|\frac{z^2}{4}\right| \leqslant \frac{m}{3}.
  \end{equation}
  Let $a = -\mu,\,c = m+\frac{1}{2}$ or $c = m+\frac{3}{2},\,x=\frac{z^2}{4}$,
  from \eqref{eq:0505-39} we know $|x / c| \leqslant \frac{1}{3}$, hence
  \begin{equation}\label{eq:0505-40}
    \begin{aligned}
    \left|\frac{a \cdots(a+k-1)}{c \cdots(c+k-1)} \frac{x^k}{k!}\right| 
    & =  \left|\frac{(-\mu) (-\mu+1) \cdots (-\mu+k-1)}{1 \cdot 2  \cdots  k}\right| \cdot\left|\frac{x}{c}\right| \cdot\left|\frac{x}{c+1}\right|  \ldots \left|\frac{x}{c+k-1}\right| \\
    & \leqslant  \frac{1}{3^k}.
    \end{aligned}
  \end{equation}
  From \eqref{eq:0505-40} we know that for $z=\xi-i s$ with $|\xi|,|s| \leqslant 2 B=2 M=\sqrt{\frac{N}{3}}$
  and $m=\left[\frac{n}{2}\right]$ with $n>N$, the following inequality holds:
  \begin{equation}\label{eq:0505-41}
    \left|\phi\left(-\mu, m+\frac{1}{2} ; \frac{z^2}{4}\right)\right|,\left|\phi\left(-\mu, m+\frac{3}{2} ; \frac{z^2}{4}\right)\right| \leqslant \sum_{k\geqslant 0} \frac{1}{3^k} = \frac{3}{2}.
  \end{equation}
  When $n>N$, combining \eqref{eq:0505-36}, \eqref{eq:0505-37} and
  \eqref{eq:0505-41} yields
  \begin{equation}\label{eq:0505-42}
    \left|c_n\right|^2 \lesssim e^{-\frac{1}{2} \xi^2-\frac{1}{2} s^2} \cdot \frac{\Gamma\left(\left[\frac{n+1}{2}\right]+\mu+\frac{1}{2}\right)}{\left[\frac{n}{2}\right]!} \cdot \frac{1}{\left\{\Gamma\left(\left[\frac{n+1}{2}\right]+\frac{1}{2}\right)\right\}^2} \cdot \left(\frac{|z|}{2}\right)^{2 n}.
  \end{equation}
  Further applying Stirling's formula yields
  \begin{equation}\label{eq:0505-43}
    \left|c_n\right|^2 \lesssim_\mu e^{-\frac{1}{2}\left(\xi^2+s^2\right)} \cdot n^\mu \cdot \frac{1}{n!} \cdot\left(\frac{\xi^2+s^2}{2}\right)^n.
  \end{equation}
  From \eqref{eq:0505-35} and \eqref{eq:0505-43} we know: when $|\xi|,|s| \leqslant 2 B=2 M=\sqrt{\frac{N}{3}}$,
  \begin{equation}\label{eq:0505-44}
    \begin{aligned}
    \left\|g_{\xi, s}-\widehat{\Pi}_N g_{\xi, s}\right\|_\omega^2 & =\sum_{n>N}\left|c_n\right|^2 \\
    & \lesssim_\mu \sum_{n>N} e^{-\frac{1}{2}\left(\xi^2+s^2\right)} \cdot n^\mu \cdot \frac{1}{n!} \cdot\left(\frac{\xi^2+s^2}{2}\right)^n \\
    & \lesssim_\mu \frac{\left(\frac{\xi^2+s^2}{2}\right)^{N+1}}{(N+1)!} \cdot N^\mu.
    \end{aligned}
  \end{equation}
  From \eqref{eq:0505-34} we have
  \begin{equation}\label{eq:0505-45}
    u_M^B-\widehat{\Pi}_N u_M^B=\frac{1}{2 \pi} \int_{|\xi| \leqslant 2 B,|s| \leqslant 2 M} \mathcal{F}\left[u^B\right](\xi)\left(g_{\xi, s}-\widehat{\Pi}_N g_{\xi, s}\right) d \xi d s.
  \end{equation}
  Applying the triangle inequality yields
  \begin{equation}\label{eq:0505-46}
    \left\|u_M^B-\widehat{\Pi}_N u_M^B\right\|_\omega \leqslant \frac{1}{2 \pi} \int_{|\xi| \leqslant 2 B,|s| \leqslant 2 M}\left|\mathcal{F} \left[ u^B \right](\xi)\right|\left\|g_{\xi, s}-\widehat{\Pi}_N g_{\xi,s}\right\|_\omega.
  \end{equation}
  Further applying the Cauchy-Schwarz inequality yields
  \begin{equation}\label{eq:0505-47}
    \begin{aligned}
    \left\|u_M^B-\widehat{\Pi}_N u_M^B\right\|_\omega^2 & \lesssim\left(\int_{|\xi| \leqslant 2 B,|s| \leqslant 2 M}\left|\mathcal{F}\left[u^B\right](\xi)\right|^2 d \xi d s\right) 
    \times\left(\int\left\|g_{\xi, s}-\widehat{\Pi}_N g_{\xi, s}\right\|_\omega^2 d \xi d s\right) \\
    & \lesssim 4 M\|u\|^2 \int_{|\xi| \leqslant 2 B,|s| \leqslant 2 M}\left\|g_{\xi, s}-\widehat{\Pi}_N g_{\xi, s}\right\|_\omega^2.
    \end{aligned}
  \end{equation}
  From \eqref{eq:0505-44} we have
  \begin{equation}\label{eq:0505-48}
    \begin{aligned}
    & \int_{|\xi| \leqslant 2 B,|s| \leqslant 2 M}\left\|g_{\xi, s}-\widehat{\Pi}_N g_{\xi, s}\right\|_\omega^2 d \xi d s \\
    &\lesssim_\mu \int_{|\xi| \leqslant 2 B,|s| \leqslant 2 M} \frac{1}{(N+1)!} \cdot\left(\frac{\xi^2+s^2}{2}\right)^{N+1} \cdot N^\mu d \xi d s \\
    & \leqslant \frac{N^\mu}{(N+1)!} \cdot \int_{\left(\xi^2+s^2\right) \leqslant 4\left(B^2+M^2\right)}\left(\frac{\xi^2+s^2}{2}\right)^{N+1} d \xi d s \\
    &= \frac{N^\mu}{(N+1)!} \cdot 2 \pi \cdot 2^{-(N+1)} \cdot \frac{N^{N+2}}{2 N+4} \cdot\left(\frac{2}{3}\right)^{N+2} \\
    &\lesssim_\mu N^{\mu-\frac{1}{2}} e^{-\frac{N}{12}}.
    \end{aligned}
  \end{equation}
  Putting \eqref{eq:0505-48} back into \eqref{eq:0505-47} we know
  $E_H$ defined in \eqref{eq:0505-9} satisfies:
  \begin{equation}\label{eq:0505-49}
    \begin{aligned}
    E_H&=\left\|u_M^B-\widehat{\Pi}_N u_M^B\right\|_\omega \\
    &\lesssim_\mu e^{-\frac{N}{24}} M^\mu \|u\|.
    \end{aligned}
  \end{equation}
  Combined the estimation of $E_s$ in \eqref{eq:0505-16}, the estimation
  of $E_f$ in \eqref{eq:1111-9}, the estimation of $E_H$ in \eqref{eq:0505-49}
  with \eqref{eq:0505-9}, we prove \eqref{eq:bigthm} under the condition $\beta = 1$.
  As discussed before, by further applying \eqref{eq:0505-2}, we can prove the general
  case when $\beta \neq 1$, this completes our proof.
\end{proof}

Combining Lemma \ref{lem:equiv} with Theorem \ref{thm:bigthm} yields the 
error of the projection onto the space spanned by the first 
$N+1$ generalized Laguerre functions:
\begin{corollary}\label{cor:Laguerre-proj}
  For $v(x)$ defined on $\mathbb{R}_+$, define $u(x)$
  on $\mathbb{R}$ by $u(x)= v(x^2)$. Let $M = \sqrt{\frac{N}{6\beta}},\, B = \sqrt{\frac{\beta N}{6}}$,
  if $u(x) \cdot (1+x^2)^{\frac{\mu}{2}+\frac{1}{4}} \in L^2\left(\mathbb{R}\right)$, then
  \begin{equation}\label{eq:1111-10}
    \begin{aligned}
    \left\|v-\widehat{\Pi}_N^{(\mu, \beta, L)} v\right\|_{x^\mu} & \lesssim_\mu\left\|u \cdot \mathbb{I}_{\{|x|>M\}}\right\|_{|x|^{2 \mu + 1}} \\
    &{}\quad +B^{-\mu}\left\|\mathcal{F}[u](B \xi)\right\|_{H^{\mu+\frac{1}{2}}(\mathbb{R} \backslash [-1,1])} \\
    &{}\quad +e^{-\frac{N}{24}}\left(\|u\|_{|x|^{2 \mu+1}}+M^{\mu+\frac{1}{2}}\|u\|\right).
    \end{aligned}
  \end{equation}
\end{corollary}

Next, we extend the projection error estimates to cover a broader range of cases, including those involving derivatives, interpolation errors, quadrature errors, and error estimates for model equations.

\section{Estimates of various approximation errors}
\label{sec:general}
The objective of this section is to demonstrate that the aforementioned projection error estimates can be extended to a wide range of approximation error estimates.

Since extending the projection error to these results follows a proof framework already established in the literature, we place the lengthy proofs in the appendix to enhance readability. 

\subsection{Projection error with derivatives}
The following theorem shows how to get estimates of projection error with derivatives from the convergence rate of projection approximation, which can be obtained from Theorem \ref{thm:bigthm}.
\begin{theorem}\label{thm:0826-2}
  Assume that
  \begin{equation}\label{eq:0826-18}
    \left\|u-\widehat{\Pi}_N^{(\mu, \beta)} u\right\|_{|x|^{2 \mu}} \lesssim_\mu\left(\int_N^{+\infty}|f(x)|^2 d x\right)^{1 / 2},
  \end{equation}
  then
  \begin{equation}\label{eq:0826-19}
    \left\|\partial_x^l\left(u-\widehat{\Pi}_N^{(\mu, \beta)} u\right)\right\|_{|x|^{2 \mu}} \lesssim_{\mu, l} \beta^l \left(\int_{N}^{+\infty} \left|f(x)\right|^2 \cdot x^l d x\right)^{1 / 2}.
  \end{equation}
\end{theorem}
See Apppendix~\ref{appendix:A} for its proof.

\begin{remark}
  Theorem \ref{thm:0826-2} can be intuitively understood as follows.
  Assume 
  \begin{displaymath}
    u = \sum_n c_n  \beta^{\mu+\frac{1}{2}} \widehat{H}_n^{(\mu)}(\beta x),
  \end{displaymath}
  then 
  \begin{displaymath}
    \bigl\|u-\widehat{\Pi}_N^{(\mu, \beta)} u\bigr\|_{|x|^{2 \mu}} = \sum_n |c_n|^2.
  \end{displaymath}
  Combining with \eqref{eq:0826-18}, it can be assumed $c_n \approx f(n)$.
  Moreover, according to Lemma~\ref{lem:derivative recurrence}, $\partial_x \widehat{H}_n^{(\mu)}(\beta x)$
  can be expressed as a summation containing terms like $\beta \sqrt{n+1} \widehat{H}_{n+1}^{(\mu)}(\beta x)$,
  hence $\partial_x u$ contains terms like $\beta \sqrt{n+1} c_{n+1} \times \left(\beta^{\mu+\frac{1}{2}} \widehat{H}_{n+1}^{(\mu)}(\beta x)\right)$.
  In other words, a single $\partial_x$ amplifies $c_n$ to $\beta \sqrt{n} c_n$, and since
  $c_n \approx f(n)$, this implies that $f(x)$ is amplified to $\beta \sqrt{x} f(x)$. 
  This is why, in \eqref{eq:0826-19}, after applying differentiation 
  $l$ times to the projection error, the error estimate is multiplied by an amplification factor of
  $\left(\beta \sqrt{x}\right)^l$.
\end{remark}

\begin{remark}
  When the projection error exhibits exponential convergence, i.e.,
  when $f(x) = e^{-c|x|^\alpha}$, from \eqref{eq:0826-18} we have 
  \begin{displaymath}
    \left\|u-\widehat{\Pi}_N^{(\mu, \beta)} u\right\|_{|x|^{2 \mu}} \lesssim_\mu e^{-cN^\alpha}.
  \end{displaymath}
  In this case, differentiation does not affect the convergence rate of the error, as can be seen from
  \eqref{eq:0826-19}, 
  \begin{displaymath}
    \left\|\partial_x^l \left(u-\widehat{\Pi}_N^{(\mu, \beta)} u\right)\right\|_{|x|^{2 \mu}} \lesssim_{\mu,\beta,l} e^{-cN^\alpha}.
  \end{displaymath}
  However, when the projection error exhibits only an algebraic convergence rate, i.e.,
  when $f(x) = |x|^{-h}$, from \eqref{eq:0826-18} we have
  \begin{displaymath}
    \left\|u-\widehat{\Pi}_N^{(\mu, \beta)} u\right\|_{|x|^{2 \mu}} \lesssim_\mu N^{-h+\frac{1}{2}}.
  \end{displaymath}
  Under this condition, if $\beta$ is a fixed constant, it follows from 
  \eqref{eq:0826-19} that each additional differentiation reduces the convergence order by
  $\frac{1}{2}$, i.e.,
  \begin{displaymath}
    \left\|\partial_x^l \left(u-\widehat{\Pi}_N^{(\mu, \beta)} u\right)\right\|_{|x|^{2 \mu}} \lesssim_{\mu,\beta,l} N^{-h+\frac{1+l}{2}}.
  \end{displaymath}
  This is consistent with the conclusions from classical theory for functions with finite smoothness, indicating that our results form a more general framework that encompasses the classical theory.
\end{remark}

For the Laguerre case, let $u(x) = v(x^2) = v(y),\, \tilde{\partial}_y = 2\sqrt{y}\partial_y$, thanks to the equivalence
\begin{equation}\label{eq:0627-2}
  \left\|\partial_x^l\left(u-\widehat{\Pi}_{N}^{(\mu,\beta)} u\right)\right\|_{|x|^{2 \mu}}=\left\|\tilde{\partial}_y^l\left(v-\widehat{\Pi}_{N}^{(\mu-\frac{1}{2}, \beta^2, L)} v\right)\right\|_{y^{\mu-\frac{1}{2}}},
\end{equation}
we can obtain the error estimate for $\left\|\tilde{\partial}_y^l\left(v-\widehat{\Pi}_{N}^{(\mu-\frac{1}{2}, \beta^2, L)} v\right)\right\|_{y^{\mu-\frac{1}{2}}}$
by combining Theorem \ref{thm:0826-2} and \eqref{eq:0627-2}.
\begin{corollary}
  Assume that
  \begin{equation}\label{eq:1111-11}
    \left\|v-\widehat{\Pi}_N^{(\mu, \beta, L)} v\right\|_{x^{\mu}} \lesssim_\mu\left(\int_{2N}^{+\infty}|f(x)|^2 d x\right)^{1 / 2}.
  \end{equation}  
  Let $\tilde{\partial}_x = 2\sqrt{x}\partial_x$, then
  \begin{equation}\label{eq:1111-12}
    \left\|\tilde{\partial}_x^l\left(v-\widehat{\Pi}_N^{(\mu, \beta, L)} v\right)\right\|_{x^{\mu}} \lesssim_{\mu, l} 
    \beta^{\frac{l}{2}}\left(\int_{2N}^{+\infty} \left|f(x)\right|^2 \cdot x^l d x\right)^{1 / 2}.
  \end{equation}
\end{corollary}

\subsection{Interpolation error and quadrature error}
\label{sec: interpolation and quadrature}
We begin by defining generalized Laguerre and generalized Hermite interpolation.
Let $\xi_{G, N, j}^{(\mu, \beta)}$ and $\xi_{R, N, j}^{(\mu, \beta)},\, 0\leqslant j \leqslant N$,
be the zeros of $\widehat{\mathscr{L}}^{(\mu)}_{N+1}(\beta x)$ and
$x\widehat{\mathscr{L}}_N^{(\mu+1)}(\beta x)$, respectively. 
They are arranged in ascending order. Then we can define Laguerre--Gauss and Laguerre--Gauss--Radau interpolation operator $\widehat{I}_{Z, N}^{(\mu,\beta)}: C[0,+\infty) \rightarrow \widehat{P}_N^{\beta,L}$,
where $Z=G, R$ represents Laguerre--Gauss or Laguerre--Gauss--Radau as ($\widehat{P}_N^{\beta,L}$ is defined in \eqref{eq:def-PNHat-L})
\begin{equation}\label{eq:Laguerre interpolation def}
  \left(\widehat{I}_{Z, N}^{(\mu,\beta)} u\right) \left(\xi_{Z, N, j}^{(\mu, \beta)}\right) = u\left(\xi_{Z, N, j}^{(\mu, \beta)}\right), \quad 0\leqslant j \leqslant N.
\end{equation}
Denote $\widehat{\omega}_{Z, N, j}^{(\mu, \beta)}, 0 \leqslant j \leqslant N, Z = G,R$ the corresponding
quadrature weights such that
\begin{equation}\label{Laguerre weights def}
  \int_{0}^\infty \phi(x) x^\mu d x=\sum_{j=0}^N \phi\left(\xi_{Z, N, j}^{(\mu, \beta)}\right) \widehat{\omega}_{Z, N, j}^{(\mu, \beta)} \quad \forall \phi \in \widehat{P}_{2 N+\lambda_Z}^{\beta, L},
\end{equation}
where $\lambda_z = 1$ and $0$ for $Z=G$ and $R$ respectively.
Similarly, we can define $\omega_{Z, N, j}^{(\mu, \beta)}, 0 \leqslant j \leqslant N, Z = G,R$ such that
\begin{equation}\label{Laguerre poly weights def}
  \int_{0}^\infty \phi(x) e^{-\beta x} x^\mu d x=\sum_{j=0}^N \phi\left(\xi_{Z, N, j}^{(\mu, \beta)}\right) \omega_{Z, N, j}^{(\mu, \beta)} \quad \forall \phi \in P_{2 N+\lambda_Z},
\end{equation}
where $\lambda_z = 1$ and $0$ for $Z=G$ and $R$ respectively, $P_k$ denotes polynomials of degree at most $k$.
From \eqref{Laguerre weights def} and \eqref{Laguerre poly weights def}, it is easy to see
\begin{equation}\label{eq:0903-12}
  \widehat{\omega}_{Z, N, j}^{(\mu, \beta)}=e^{\beta \xi_{Z, N, j}^{(\mu, \beta)}} \omega_{Z, N, j}^{(\mu, \beta)}.
\end{equation}

As for the generalized Hermite interpolation, let $\left\{x_{N,j}^{(\mu,\beta)}\right\}_{j=0}^N$
be the zeros of $\widehat{H}_{N+1}^{(\mu)}(\beta x)$ which are arranged in ascending order. 
We define the Hermite interpolation operator as
$\widehat{I}_{N}^{(\mu,\beta)}: C(-\infty,+\infty) \rightarrow \widehat{P}_N^{\beta,H}$,
where $\widehat{P}_N^{\beta,H}$ is defined in \eqref{eq:def-PNHat-H},
such that
\begin{equation}\label{eq:0710-3}
  \left(\widehat{I}_{N}^{(\mu,\beta)} u\right)\left(x_{N,j}^{(\mu,\beta)}\right)=u\left(x_{N,j}^{(\mu,\beta)}\right), \quad 0 \leqslant j \leqslant N.
\end{equation}
Denote $\widehat{\omega}_{N, j}^{(\mu, \beta)}, 0 \leqslant j \leqslant N, Z = G,R$ the corresponding
quadrature weights such that
\begin{equation}\label{Hermite weights def}
  \int_{-\infty}^\infty \phi(x) |x|^{2\mu} d x=\sum_{j=0}^N \phi\left(x_{N, j}^{(\mu, \beta)}\right) \widehat{\omega}_{N, j}^{(\mu, \beta)} \quad \forall \phi \in \widehat{P}_{2 N+1}^{\beta, H}.
\end{equation}

We now present the interpolation error estimates. The next theorem
shows that the generalized Hermite interpolation error can be controlled by
the projection error in some way.

\begin{theorem}\label{thm:0810-1}
  Let $\omega = |x|^{2\mu}$, for $\widehat{I}_{2 N}^{(\mu, \beta)}$, we have
  \begin{equation}\label{eq:0810-12}
    \begin{aligned}
    \left\|u-\widehat{I}_{2 N}^{(\mu, \beta)} u\right\|_{\omega} 
    & \lesssim_\mu N^{\frac{1}{2}}\left\|u-\widehat{\Pi}_{2 N}^{(\mu, \beta)} u\right\|_{\omega}+ \beta^{-1}\left\|\partial\left(u-\widehat{\Pi}_{2 N}^{(\mu, \beta)} u\right)\right\|_{\omega} \\
    &{}\quad +(\beta \sqrt{N})^{-\mu-\frac{1}{2}}\left|u(0)-\left(\widehat{\Pi}_{2 N}^{(\mu, \beta)} u\right)(0)\right|.
    \end{aligned}
  \end{equation}
  For $\widehat{I}_{2 N + 1}^{(\mu, \beta)}$, we have
  \begin{equation}\label{eq:0810-13}
    \left\|u-\widehat{I}_{2 N+1}^{(\mu, \beta)} u\right\|_{\omega} \lesssim_\mu N^{\frac{1}{2}}\left\|u-\widehat{\Pi}_{2 N+1}^{(\mu, \beta)} u\right\|_{\omega}+ \beta^{-1} \left\|\partial\left(u-\widehat{\Pi}_{2 N+1}^{(\mu, \beta)} u\right)\right\|_{\omega}.
  \end{equation}
\end{theorem}
See Appendix~\ref{appendix:B} for its proof.

A further treatment of $\Bigl\|\partial\bigl(u-\widehat{\Pi}_{2 N}^{(\mu, \beta)} u\bigr)\Bigr\|_{\omega}$
and $\Bigl|u(0)-\bigl(\widehat{\Pi}_{2 N}^{(\mu, \beta)} u\bigr)(0)\Bigr|$ in Theorem~\ref{thm:0810-1} yields the following result.
\begin{theorem}\label{thm:0916-1}
  Let $\omega = |x|^{2\mu}$. Assume that
  \begin{equation}\label{eq:0916-1}
    \left\|u-\widehat{\Pi}_N^{(\mu, \beta)} u\right\|_{\omega} \lesssim_\mu\left(\int_N^{+\infty} \left|f(x)\right|^2 d x\right)^{1 / 2}.
  \end{equation}
  Then, for $\widehat{I}_{2 N}^{(\mu, \beta)} u$, we have
  \begin{equation}\label{eq:0916-2}
    \left\|u-\widehat{I}_{2 N}^{(\mu, \beta)} u\right\|_\omega \lesssim_\mu\left(\int_{2 N}^{+\infty} \left|f(x)\right|^2 \cdot \left(x + \beta^{-2\mu}\left(\frac{x}{N}\right)^{\mu+1}\right) d x\right)^{1 / 2},
  \end{equation}
  for $\widehat{I}_{2 N+1}^{(\mu, \beta)} u$, we have
  \begin{equation}\label{eq:0916-3}
    \left\|u-\widehat{I}_{2 N+1}^{(\mu, \beta)} u\right\|_\omega \lesssim_\mu \quad\left(\int_{2 N+1}^{+\infty} \left|f(x)\right|^2 \cdot x d x\right)^{1 / 2}.
  \end{equation}
\end{theorem}

\begin{remark}
  Similar to classical theory, the interpolation error includes an additional factor compared to the projection error, which stems from the stability estimate of interpolation.
  The results we present here are more general. The classical theory only addresses target functions with finite smoothness, which corresponds to $f(x) = |x|^{-h}$.
  In this case, from \eqref{eq:0916-1} we have
  \begin{displaymath}
    \left\|u-\widehat{\Pi}_N^{(\mu, \beta)} u\right\|_{\omega} \lesssim_\mu N^{-h+\frac{1}{2}}.
  \end{displaymath}
  For fixed $\beta$, applying \eqref{eq:0916-2} and \eqref{eq:0916-3} yields
  \begin{displaymath}
    \left\|u-\widehat{I}_{N}^{(\mu, \beta)} u\right\|_\omega \lesssim_{\mu,\beta} N^{-h+1}.
  \end{displaymath}
  That is, the estimate of the interpolation error includes an additional factor 
  $N^{\frac{1}{2}}$ compared to that of the projection error.
  This aligns with the classical results~\cite{aguirre_hermite_2005}, although we do not focus on minimizing this additional factor as much as possible here.
  Because our primary focus is on extending the new projection error estimates to the interpolation error case, 
  thereby demonstrating that interpolation error exhibits behavior similar to projection error: both can be controlled by spatial/frequency truncation error and spectral error.
\end{remark}

\begin{proof}
  To prove \eqref{eq:0916-2}, we need to handle the three terms on the right-hand side of \eqref{eq:0810-12}.
  For the first term, from \eqref{eq:0916-1} we know
  \begin{equation}\label{eq:0916-5}
    N^{\frac{1}{2}}\left\|u-\widehat{\Pi}_{2 N}^{(\mu, \beta)} u\right\|_\omega \lesssim_\mu \quad\left(\int_{2 N}^{+\infty} \left|f(x)\right|^2 \cdot x d x\right)^{1 / 2}.
  \end{equation}
  For the second term, applying Theorem \ref{thm:0826-2} yields
  \begin{equation}\label{eq:0916-6}
    \beta^{-1}\left\|\partial\left(u-\widehat{\Pi}_{2 N}^{(\mu, \beta)} u\right)\right\|_\omega \lesssim_\mu \quad\left(\int_{2 N}^{+\infty} \left|f(x)\right|^2 \cdot x d x\right)^{1 / 2}.
  \end{equation}
  For the third term, let $u(x) = \sum_n c_n \widehat{H}_n^{(\mu)}(\beta x)$,
  then
  \begin{equation}\label{eq:0916-6.5}
    u(0)-\left(\widehat{\Pi}_{2 N}^{(\mu, \beta)} u\right)(0) = \sum_{n > N} c_{2n} \widehat{H}_{2n}^{(\mu)}(0).
  \end{equation}
  Applying equation (7.6) of~\cite{shen2011spectral}, \eqref{eq:relaion-Laguerre-Hermite} and \eqref{eq:Hermite-func} yields
  \begin{equation}\label{eq:0916-7}
    \widehat{H}_{2 n}^{(\mu)}(0)=\frac{(-1)^n}{\Gamma\left(\mu+\frac{1}{2}\right)} \sqrt{\frac{\Gamma\left(n+\mu+\frac{1}{2}\right)}{n!}} \sim_\mu n^{\frac{1}{2}\left(\mu-\frac{1}{2}\right)}.
  \end{equation}
  Hence
  \begin{equation}\label{eq:0916-8}
    \begin{aligned}
    \left|u(0)-\bigl(\widehat{\Pi}_{2 N}^{(\mu, \beta)} u\right)(0)\bigr|^2 
    & =\Bigl|\sum_{n>N} c_{2 n} \widehat{H}_{2 n}^{(\mu)}(0)\Bigr|^2 \\
    & \lesssim_\mu\Bigl(\sum_{n>N}\left|c_{2 n}\right|^2 \cdot n^{\mu+1}\Bigr)\Bigl(\sum_{n>N} n^{-\frac{3}{2}}\Bigr) \\
    & \lesssim_\mu N^{-\frac{1}{2}} \sum_{n > 2N}\left|c_n\right|^2 \cdot n^{\mu+1}.
    \end{aligned}
  \end{equation}
  Combining \eqref{eq:0916-1} with $u(x) = \sum_n c_n \widehat{H}_n^{(\mu)}(\beta x)$ we have
  \begin{equation}\label{eq:0916-9}
    \sum_{n>N}\left|c_n\right|^2 \lesssim_\mu \beta \cdot \int_N^{+\infty} \left|f(x)\right|^2 d x.
  \end{equation}
  By applying Abel's transformation and \eqref{eq:0916-9}, one can prove that
  \begin{equation}\label{eq:0916-10}
    \sum_{n>2 N}\left|c_n\right|^2 \cdot n^{\mu+1} \lesssim_\mu \beta \cdot \int_{2 N}^{+\infty} \left|f(x)\right|^2 \cdot x^{\mu+1} d x.
  \end{equation}
  Substituting \eqref{eq:0916-10} back into \eqref{eq:0916-8} yields
  \begin{equation}\label{eq:0916-15}
    \Bigl|u(0)-\left(\widehat{\Pi}_{2 N}^{(\mu, \beta)} u\right)(0)\Bigr| 
    \lesssim N^{-\frac{1}{4}} \beta^{\frac{1}{2}}\left(\int_{2 N}^{+\infty} \left|f(x)\right|^2 \cdot x^{\mu+1} d x\right)^{1 / 2}.
  \end{equation}
  Combining \eqref{eq:0916-5}, \eqref{eq:0916-5}, \eqref{eq:0916-15} with
  \eqref{eq:0810-12}, we prove \eqref{eq:0916-2}. The proof for \eqref{eq:0916-3}
  is analogous.
\end{proof}

Let $u(x) = v(x^2) = v(y)$,
from the relationship between Laguerre and Hermite function \eqref{eq:relaion-Laguerre-Hermite} it is not 
difficult to deduce
\begin{equation}\label{eq:0624-17}
  \begin{aligned}
  \left\|\left(u-\widehat{I}_{2 N+1}^{(\mu,\beta)} u\right)\right\|_{|x|^{2\mu}}&=\left\|\left(v-\widehat{I}_{G,N}^{(\mu-\frac{1}{2},\beta^2)} v\right)\right\|_{y^{\mu-\frac{1}{2}}}, \\
  \left\|\left(u-\widehat{I}_{2 N}^{(\mu,\beta)} u\right)\right\|_{|x|^{2 \mu}}&=\left\|\left(v-\widehat{I}_{R,N}^{(\mu-\frac{1}{2},\beta^2)} v\right)\right\|_{y^{\mu-\frac{1}{2}}}.
  \end{aligned}
\end{equation}
Combining estimates for generalized Hermite interpolation error in Theorem~\ref{thm:0916-1} with \eqref{eq:0624-17} yields the following Laguerre interpolation error estimate.
\begin{corollary}
  Assume that
  \begin{equation}\label{eq:1111-13}
    \left\|v-\widehat{\Pi}_N^{(\mu, \beta, L)} v\right\|_{x^\mu} \lesssim_\mu\left(\int_{2N}^{+\infty} \left|f(x)\right|^2 d x\right)^{1 / 2}.
  \end{equation}
  Then, for Gauss--Radau interpolation $\widehat{I}_{R, N}^{(\mu, \beta)} v$, we have
  \begin{equation}\label{eq:1111-15}
    \left\|v-\widehat{I}_{R, N}^{(\mu, \beta)} v\right\|_{x^\mu} \lesssim_\mu\left(\int_{2 N}^{+\infty} \left|f(x)\right|^2 \cdot \left(x + \beta^{-\mu-\frac{1}{2}}\left(\frac{x}{N}\right)^{\mu+\frac{3}{2}}\right) d x\right)^{1 / 2}.
  \end{equation}
  For Gauss interpolation $\widehat{I}_{G,N}^{(\mu, \beta)} v$, we have
  \begin{equation}\label{eq:1111-14}
    \left\|v-\widehat{I}_{G,N}^{(\mu, \beta)} v\right\|_{x^\mu} \lesssim_\mu \quad\left(\int_{2 N}^{+\infty} \left|f(x)\right|^2 \cdot x d x\right)^{1 / 2}.
  \end{equation}
\end{corollary}

\subsection{Numerical quadrature error}
We demonstrate that for general Gauss quadrature, the quadrature error can be transformed into  
interpolation errors for estimation, thereby enabling the application of the 
previously established scaled Hermite/Laguerre interpolation error estimates to 
derive corresponding quadrature error estimates.

Specifically, consider a weighted integral over $\mathcal{X}$
\begin{equation}\label{eq:0113-1}
  Q(\phi) = \int_{\mathcal{X}} \phi(x) \omega(x) dx.
\end{equation}
The corresponding Gauss quadrature is defined as
\begin{equation}\label{eq:0113-2}
  Q_N(\phi) = \sum_{j=0}^{N} \omega_{N,i}\phi(x_{N,i}),
\end{equation}
where $x_{N,i}$ and $\omega_{N,i}$ are the numerical quadrature nodes and weights, respectively.
Next, define the interpolation operator at the quadrature nodes as $I_{N}: C(\mathcal{X}) \rightarrow P_N$
satisfying
\begin{equation}\label{eq:0113-3}
  I_N[\phi](x_{N,i}) = \phi(x_{N,i}),\ \forall\, 0 \leqslant i \leqslant N.
\end{equation}
Then, we have the following result.
\begin{lemma}\label{lem:quadrature}
  Assuming the numerical quadrature \eqref{eq:0113-2} satisfies the following exactness property:
  \begin{equation}\label{eq:0113-4}
    Q_N(\phi) = Q(\phi),\, \forall\, \phi \in P_{2N}.
  \end{equation}
 Let $\phi=\phi_1 \phi_2$, $\psi_1=I_{N} \phi_1, \,\psi_2=I_{N} \phi_2$.
  Then the quadrature error can be controlled in terms of the interpolation error
  $\phi_1-\psi_1,\,\phi_2-\psi_2$, i.e.,
  \begin{equation}\label{eq:0113-5}
    \begin{aligned}
    \left|Q(\phi)-Q_{N}(\phi)\right| & \leqslant\left(\left\|\phi_1\right\|_\omega+\left\|\phi_2\right\|_\omega+\left\|\psi_1\right\|_\omega+\left\|\psi_2\right\|_\omega\right) \\
    &\quad \times\left(\left\|\phi_1-\psi_1\right\|_\omega+\left\|\phi_2-\psi_2\right\|_\omega\right).
    \end{aligned}
  \end{equation}
\end{lemma}

\begin{proof}
  From the definition of $I_{N}$ \eqref{eq:0113-3}, we know $\psi_i\left(x_{N,j}\right)=\phi_i\left(x_{N,j}\right), \,0 \leqslant j \leqslant N, \,i=1,2$.
  Since $\psi_1 \psi_2 \in P_{2 N}$, thanks to the
  exactness,
  \begin{equation}\label{eq:0925-4}
    \begin{aligned}
    Q\left(\psi_1 \psi_2\right) & =\int_{\mathcal{X}} \omega(x) \psi_1 \psi_2 d x \\
    & =\sum_{j=0}^N\left(\psi_1 \psi_2\right)\left(x_{N,j}\right) \omega_{N,j} \\
    & =\sum_{j=0}^N\left(\phi_1 \phi_2\right)\left(x_{N,j}\right) \omega_{N,j} \\
    & =Q_{N}(\phi).
    \end{aligned}
  \end{equation}
  Hence
  \begin{equation}\label{eq:0925-5}
    \begin{aligned}
    \left|Q(\phi)-Q_{N}(\phi)\right| & =\left|Q\left(\phi_1 \phi_2\right)-Q\left(\psi_1 \psi_2\right)\right| \\
    &\quad \leqslant \left|Q\left(\left(\phi_1-\psi_1\right) \phi_2\right)\right|+\left|Q\left(\psi_1\left(\phi_2-\psi_2\right)\right)\right|.
    \end{aligned}
  \end{equation}
  By applying the Cauchy--Schwarz inequality, we obtain
  \begin{equation}\label{eq:0925-5.3}
    \begin{aligned}
    \left|Q\left(\left(\phi_1-\psi_1\right) \phi_2\right)\right| & =\left|\int_{\mathcal{X}}\left(\phi_1-\psi_1\right) \phi_2 \omega(x) d x\right| \\
    &\quad \leqslant\left\|\phi_1-\psi_1\right\|_\omega\left\|\phi_2\right\|_\omega.
    \end{aligned}
  \end{equation}
  Similarly,
  \begin{equation}\label{eq:0925-5.6}
    \left|Q\left(\psi_1\left(\phi_2-\psi_2\right)\right)\right| \leqslant\left\|\psi_1\right\|_\omega\left\|\phi_2-\psi_2\right\|_\omega.
  \end{equation}
  Combining \eqref{eq:0925-5.3}, \eqref{eq:0925-5.6} with \eqref{eq:0925-5}
  yields \eqref{eq:0113-5}.
\end{proof}

\subsubsection{Laguerre quadrature error}
We denote the integral $\int_0^{\infty} \phi(x) x^\mu d x$
by $Q^{(\mu)}(\phi)$, and its approximation using Laguerre quadrature
$\sum_{j=0}^N \phi\left(\xi_{Z, N, j}^{(\mu, \beta)}\right) \widehat{\omega}_{Z, N, j}^{(\mu, \beta)}$
by $Q_{Z, N}^{(\mu, \beta)}(\phi)$. Then from \eqref{Laguerre weights def}
we have
\begin{equation}\label{eq:0925-2}
  Q^{(\mu)}(\phi)=Q_{Z, N}^{(\mu, \beta)}(\phi), \quad \forall\, \phi \in \widehat{P}_{2 N+\lambda z}^{\beta, L}.
\end{equation}
Prepared with the preceding notation, we now state the following theorem on the estimation of Laguerre quadrature error,
which can be derived from Theorem~\ref{lem:quadrature}.

\begin{theorem}\label{thm:Laguerre quadrature error}
  Assume $\phi=\phi_1 \phi_2$, let $\psi_1=\widehat{I}_{Z, N}^{(\mu, \beta)} \phi_1, \,\psi_2=\widehat{I}_{Z, N}^{(\mu, \beta)} \phi_2, \,\omega=x^\mu$.
  Then Laguerre quadrature error can be controlled in terms of the interpolation error
  $\phi_1-\psi_1,\,\phi_2-\psi_2$, more precisely,
  \begin{equation}\label{eq:0925-3}
    \begin{aligned}
    \left|Q^{(\mu)}(\phi)-Q_{Z, N}^{(\mu, \beta)}(\phi)\right| & \leqslant\left(\left\|\phi_1\right\|_\omega+\left\|\phi_2\right\|_\omega+\left\|\psi_1\right\|_\omega+\left\|\psi_2\right\|_\omega\right) \\
    &\quad \times\left(\left\|\phi_1-\psi_1\right\|_\omega+\left\|\phi_2-\psi_2\right\|_\omega\right).
    \end{aligned}
  \end{equation}
\end{theorem}

\subsubsection{Hermtie quadrature error}
We denote the integral $\int_{-\infty}^{\infty} \phi(x)|x|^{2 \mu} d x$
by $Q_H^{(\mu)}(\phi)$, and its approximation using Hermite quadrature
$\sum_{j=0}^N \phi\left(x_{N, j}^{(\mu, \beta)}\right) \widehat{\omega}_{N, j}^{(\mu, \beta)}$
by $Q_{H, N}^{(\mu, \beta)}(\phi)$. Then from \eqref{Hermite weights def}
we have
\begin{equation}\label{eq:0926-7}
  Q_H^{(\mu)}(\phi)=Q_{H, N}^{(\mu, \beta)}(\phi), \quad \forall\, \phi \in \widehat{P}_{2 N+1}^{\beta, H}.
\end{equation}
Similar to Theorem~\ref{thm:Laguerre quadrature error}, we have
\begin{theorem}\label{thm:Hermite quadrature error}
  Assume $\phi=\phi_1 \phi_2$, let $\psi_1=\widehat{I}_{N}^{(\mu, \beta)} \phi_1, \,\psi_2=\widehat{I}_{N}^{(\mu, \beta)} \phi_2, \,\omega=|x|^{2\mu}$.
  Then the error of generalized Hermite quadrature can be controlled in terms of the interpolation error
  $\phi_1-\psi_1,\,\phi_2-\psi_2$, more precisely,
  \begin{equation}\label{eq:0926-8}
    \begin{aligned}
    \left|Q_H^{(\mu)}(\phi)-Q_{H, N}^{(\mu, \beta)}(\phi)\right| & \leqslant\left(\left\|\phi_1\right\|_\omega+\left\|\phi_2\right\|_\omega+\left\|\psi_1\right\|_\omega+\left\|\psi_2\right\|_\omega\right) \\
    & \times\left(\left\|\phi_1-\psi_1\right\|_\omega+\left\|\phi_2-\psi_2\right\|_\omega\right).
    \end{aligned}
  \end{equation}
\end{theorem}

\subsubsection{Application: optimality of scaled Hermite/Laguerre quadrature}
\label{subsec:optimality}
In the literature, Hermite and Laguerre quadratures have been found to be less efficient compared to quadratures on bounded domains.
For instance, Trefethen~\cite{trefethen2022exactness} pointed out that when computing
\begin{equation}\label{eq:0113-6}
  \int_{-\infty}^{+\infty} e^{-x^2} cos(x^3) dx,
\end{equation}
Gauss--Legendre, Clenshaw--Curtis, and trapezoidal
quadrature can achieve a convergence rate of $\exp(-CN^{2/3})$, while numerical result shows
Gauss--Hermite quadrature can only achieve $\exp(-C\sqrt{N})$ rate.

However, such a comparison is not entirely fair. 
Quadrature rules on bounded domains are typically applied to a properly truncated region, 
whereas Hermite quadrature does not distribute its quadrature points over the same region through scaling.
Once this is achieved, by applying the error analysis established in this paper, 
it can be demonstrated that Hermite quadrature will attain the same level of accuracy as quadrature rules on bounded domains.

Specifically, let $\phi_1(x) = e^{-\frac{x^2}{2}}$, $\phi_2(x) = e^{-\frac{x^2}{2}} \cos(x^3)$, 
by Theorem \ref{thm:Hermite quadrature error}, the Hermite quadrature error when computing
\eqref{eq:0113-6} can be controlled by $\bigl\|\phi_1-\widehat{I}_N^{(\beta)} \phi_1\bigr\|$ and $\bigl\|\phi_2-\widehat{I}_N^{(\beta)} \phi_2\bigr\|$.
Furthermore, using Theorem \ref{thm:0916-1}, it can be concluded that 
the interpolation error exhibits behavior similar to the projection error. 
It is only necessary to analyze the optimal scaling factor for the projection error.

By the Paley--Winner theorem~\cite{paley1934fourier},
\begin{equation}\label{eq:0113-7}
  \left|\mathcal{F}[\phi_2](\xi)\right| \lesssim e^{-c|\xi|}.
\end{equation}
Hence, let $M = \frac{\sqrt{N}}{2\sqrt{3}}\beta,\, B = \frac{\sqrt{N}}{2\sqrt{3}}/\beta$, we have
\begin{equation}\label{eq:0113-8}
  \bigl\|\phi_2 \cdot \mathbb{I}_{\{|x|>M\}}\bigr\| \lesssim e^{-c N / \beta^2}, 
  \quad B^{\frac{1}{2}} \left\| \mathcal{F}\left[\phi_2\right](B\xi) \right\|_{\mathbb{R} \backslash[-1,1]} \lesssim e^{-c \sqrt{N} \beta}.
\end{equation}
Letting $\beta = N^{\frac{1}{6}}$ to balance the frequency and spatial truncation errors, and applying Theorem \ref{thm:bigthm}, we obtain:
\begin{equation}\label{eq:0113-9}
  \bigl\|\phi_2-\widehat{\Pi}_N^\beta \phi_2\bigr\| \lesssim e^{-cN^{\frac{2}{3}}}.
\end{equation}
Similarly, it can be concluded that the same convergence accuracy holds for $\phi_1$ as well.
Based on the above analysis, with the optimal choice of $\beta = N^{\frac{1}{6}}$, the following quadrature error estimate holds:
\begin{equation}\label{eq:0113-10}
  \left|Q_H\left(e^{-x^2}cos(x^3)\right)-Q_{H, N}^{(\beta)}\left(e^{-x^2}cos(x^3)\right)\right| \lesssim e^{-cN^{\frac{2}{3}}}.
\end{equation}
This indicates that the scaling optimized Hermite quadrature achieves accuracy comparable to quadrature rules on bounded domains.
Laguerre quadrature exhibits similar inefficiency, which can likewise be overcome by selecting an appropriate scaling factor to balance the spatial and frequency truncation errors.

\subsection{Model problem}\label{subsec:model}
Consider the model problem on the half-line
\begin{equation}\label{eq:0624-18}
  -u_{x x}+\gamma u=f, \quad x \in \mathbb{R}_{+}, \quad \gamma>0 ; \quad u(0)=0, \quad \lim _{x \rightarrow+\infty} u(x)=0 .
\end{equation}
Denote
\begin{displaymath}
  \widehat{P}_{N,\beta}^0=\bigl\{\phi \in \widehat{P}_N^{\beta,L}: \phi(0)=0\bigr\} .
\end{displaymath}
The weak formulation for \eqref{eq:0624-18} is
\begin{equation}\label{eq:0624-19}
  \left\{\begin{array}{l}
  \text { Find } u \in H_0^1\left(\mathbb{R}_{+}\right) \text {such that } \\
  a(u, v):=\left(u^{\prime}, v^{\prime}\right)+\gamma(u, v)=(f, v), \quad \forall v \in H_0^1\left(\mathbb{R}_{+}\right).
  \end{array}\right.
\end{equation}
The corresponding Laguerre--Galerkin approximation to \eqref{eq:0624-18} is
\begin{equation}\label{eq:0624-20}
  \left\{\begin{array}{l}
  \text { Find } u_{N,\beta} \in \widehat{P}_{N,\beta}^0 \text { such that } \\
  a\left(u_{N,\beta}, v_{N,\beta}\right)=\bigl(\widehat{I}_{N,\beta} f, v_{N,\beta}\bigr), 
  \quad \forall \,v_{N,\beta} \in \widehat{P}_{N,\beta}^0.
  \end{array}\right.
\end{equation}
Similar to Theorem 7.11 of~\cite{shen2011spectral}, we define projection
operator $\widehat{\Pi}_{N, \beta}^{1,0}: H_0^1\left(\mathbb{R}_{+}\right) \longrightarrow \widehat{P}_{N, \beta}^0$ satisfying
\begin{equation}\label{eq:0626-29}
  \left(\left(u-\widehat{\Pi}_{N, \beta}^{1,0} u\right)^{\prime}, v_{N, \beta}^{\prime}\right)+\frac{\beta^2}{4}\left(u-\widehat{\Pi}_{N, \beta}^{1,0} u , v_{N, \beta}\right)=0, \quad \forall v_{N, \beta} \in \widehat{P}_{N, \beta}^0.
\end{equation}
Following the proof of Theorem 7.10, Theorem 7.11 of~\cite{shen2011spectral}, 
we have
\begin{theorem}\label{thm:0626-thm2}
  For any $u \in H_0^1\left(\mathbb{R}_{+}\right)$, use $\widehat{\Pi}_N^\beta$
  to represent $\widehat{\Pi}_{N}^{(0,\beta,L)}$ defined in \eqref{eq:def-projHat-L}, we have
  \begin{equation}
    \left\|u-\widehat{\Pi}_{N, \beta}^{1,0} u\right\|_1 \leqslant\left(2+\frac{2}{\beta}\right)\left\|u^{\prime}-\widehat{\Pi}_{N-1}^\beta u^{\prime}\right\|+(\beta+1)\left\|u-\widehat{\Pi}_{N-1}^\beta u\right\|.
  \end{equation}
\end{theorem}
Further, referring to the proof of Theorem 7.19 of~\cite{shen2011spectral}, and
combining it with Theorem~\ref{thm:0626-thm2}, we have
\begin{theorem}\label{thm:0626-thm3}
  For the solution $u$ to the model problem \eqref{eq:0624-19} and
  the solution $u_{N,\beta}$ to the Laguerre--Galerkin approximation
  problem \eqref{eq:0624-20}, we have
  \begin{equation}\label{eq:0626-thm3}
    \begin{aligned}
    \left\|u-u_{N,\beta}\right\|_1 & \lesssim\left(\frac{1}{\beta}+1\right)\left\|u^{\prime}-\widehat{\Pi}_{N-1}^\beta u^{\prime}\right\| \\
    &\quad +(1+\beta)\left\|u-\widehat{\Pi}_{N-1}^\beta u\right\| \\
    &\quad +\left\|f-\widehat{I}_{N}^\beta f\right\|.
    \end{aligned}
  \end{equation}
\end{theorem}
Combining the projection and interpolation error estimates established
previously with Theorem~\ref{thm:0626-thm3}, we find that the error $\left\|u-u_{N,\beta}\right\|_1$
can be controlled by spatial/frequency truncation error and spectral error.

A similar analysis of the error for the Hermite--Galerkin method for the model problem defined on the whole real line can be found in our previous paper~\cite{hu2024scaling}.

\section{Numerical results and discussions}
\label{sec:experiments}
We verify our error estimates by comparing the numerical results with theoretical predictions. Some typical examples are selected to show the new insights gained from our estimates.

\subsection{The optimal scaling balances the spatial and frequency truncation errors}
From Theorem~\ref{thm:bigthm} we know the projection error can be controlled by spatial truncation error, frequency truncation error, and spectral error. The spectral error is bounded by $\|u\|e^{-cN}$ quasi-uniformly for the scaling factor. Hence, we only need to consider these truncation errors.

If the spatial truncation error and the frequency truncation error are imbalanced, i.e., one is much larger than the other (often by an order of magnitude), then by a proper scaling we can reduce the error of the dominant side and thus reduce the total error.

In general, finding the optimal scaling is equivalent to balancing the spatial and frequency truncation error. 
Through the above extended analysis in Section~\ref{sec:general}, we know that this principle also applies to
interpolation and Galerkin approximation. In the following examples, we will show how it works.

\subsection{Proper scaling recovers a geometric convergence}
\label{subsec: exp conv}
From (59), (60) of~\cite{boyd2014fourier}, we know that the Fourier transform of $u_H = e^{-x^{2n}}$ satisfies
\begin{displaymath}
  \begin{aligned}
  \Phi(k ; n) &\sim\left(\frac{k}{2 n}\right)^{1 /(2 n-1)} \exp \left(z \Psi\left(t_\sigma\right)\right) \frac{\sqrt{\pi}}{\sqrt{z}} \frac{1}{\sqrt{-P_2}} \\
  & =C_1 k^{(1-n) /(2 n-1)} \exp \bigl(-C_2 k^{2 n /(2 n-1)}\bigr) \cos \bigl(C_3 k^{2 n /(2 n-1)}-\xi_n\bigr),
  \end{aligned}
\end{displaymath}
where $C_1, C_2, C_3, \xi_n$ constants that depend on $n$.
If $\mu + \frac{1}{2} \in \mathbb{N}$, then for $r \in \mathbb{N},\, 0 \leqslant r \leqslant \mu + \frac{1}{2}$,  
\begin{displaymath}
  u_H(x) \lesssim e^{-x^{2n}}, \quad \partial^r \mathcal{F}[u_H](k) \lesssim_\mu e^{-c k^{2n/(2n-1)}}.
\end{displaymath}
By Theorem~\ref{cor:Laguerre-proj}, if $\beta = a$, where a is a constant, then
for $y = x^2,\,u(y) = u_H(x) = e^{-y^n}$, we have
\begin{equation}\label{eq:exp(-x2n)-beta1}
  \|u - \widehat{\Pi}_{N}^{(\mu,\beta,L)} u\|_{y^{\mu}} \lesssim_\mu e^{-c N^{n/(2n-1)}}.
\end{equation}
If $\mu + \frac{1}{2} \notin \mathbb{N}$, by interpolation inequality \eqref{eq:exp(-x2n)-beta1} also holds.
Taking $\beta = a \left(N\right)^{(n-1)/n}$ balances the spatial and frequency truncation error, we have
\begin{equation}\label{eq:exp(-x2n)-beta-opt}
  \|u - \widehat{\Pi}_{N}^{(\mu,\beta,L)} u\|_{y^{\mu}} \lesssim_\mu e^{-c N}.
\end{equation}

We now solve the model problem \eqref{eq:0624-18} by the Laguerre--Galerkin method defined in \eqref{eq:0624-20} with $\gamma = 1$ and true solution $u = e^{-x^n}-e^{-\frac{x^n}{2}},n\in \mathbb{N}^+$. 
With a constant scaling factor $\beta = 1$, the $L^2$  error $\|u - u_{N,\beta}\|$ is presented in \ref{fig:exp(-x2)-beta1}, where $N$ is taken from 10 to 300.

\begin{figure}[H]
  \centering
  \includegraphics[width=0.65\linewidth]{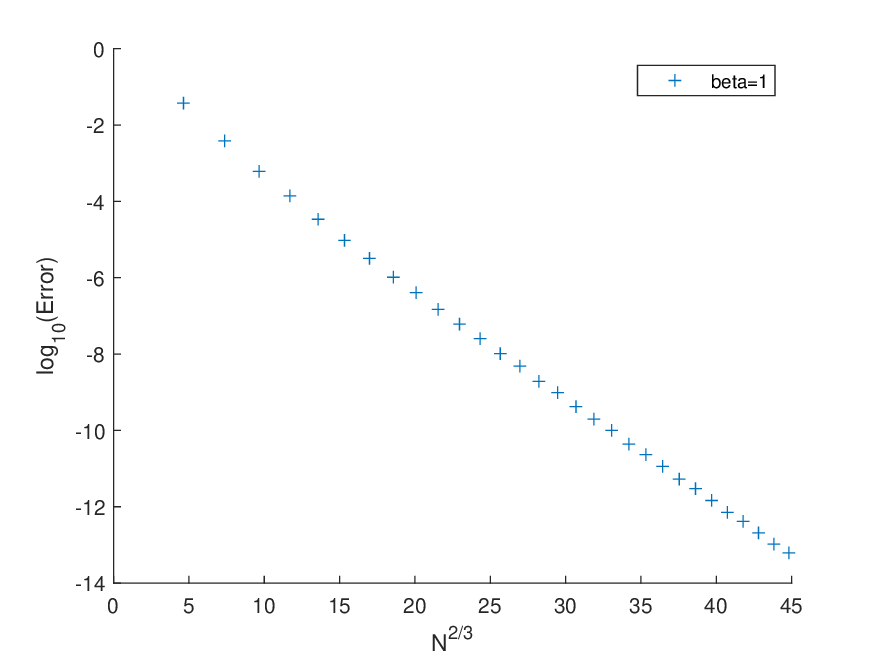}
  \caption{The $L^2$ error $\|u - u_{N,\beta}\|$ of the Laguerre--Galerkin method solving \eqref{eq:0624-18} with $u = e^{-x^2} - e^{-\frac{x^2}{2}}, \beta=1$.  }
  \label{fig:exp(-x2)-beta1}
\end{figure}

We observe the expected convergence order as in \eqref{eq:exp(-x2n)-beta1}.

The $L^2$ error $\|u - u_{N,\beta}\|$ corresponding to scaling $\beta = N^{1/2}$
is given in \ref{fig:exp(-x2)-beta-opt}, which verifies \eqref{eq:exp(-x2n)-beta-opt}.
Notice that in \ref{fig:exp(-x2)-beta1} $N$ is taken from 10 to 300 while
in \ref{fig:exp(-x2)-beta-opt} $N$ is only taken from 10 to 50,
a simple scaling significantly improves accuracy.

\begin{figure}[H]
  \centering
  \includegraphics[width=0.65\linewidth]{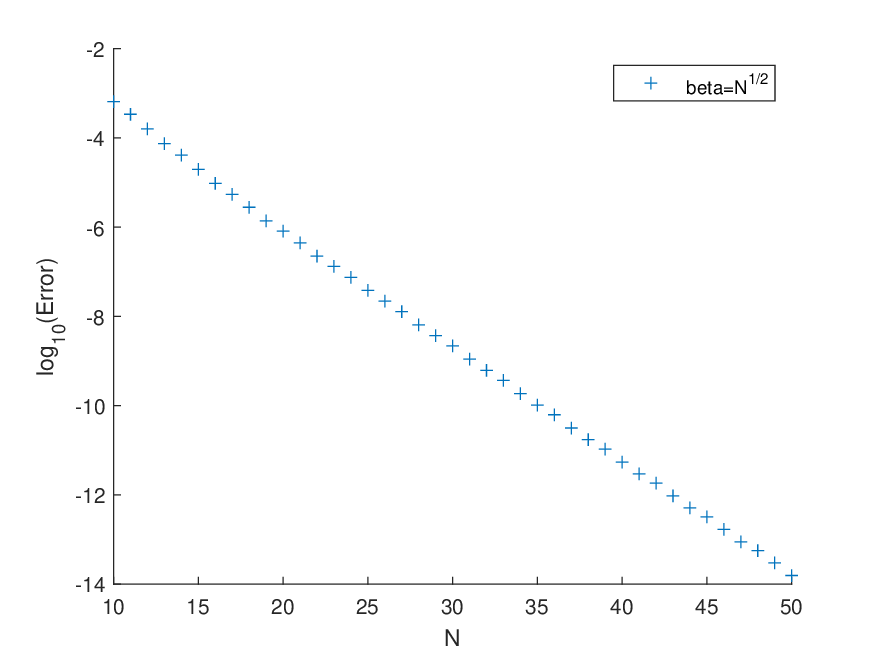}
  \caption{The $L^2$ error $\|u - u_{N,\beta}\|$ of the Laguerre--Galerkin method solving \eqref{eq:0624-18} with $u = e^{-x^2} - e^{-\frac{x^2}{2}}, \beta=N^{1/2}$.  }
  \label{fig:exp(-x2)-beta-opt}
\end{figure}

\subsection{Proper scaling doubles the convergence order}
\label{subsec: double conv}
In this subsection, we consider the approximation for
$u_H(x) = \left(1+x^2\right)^{-h}$, which stands for a class
of functions having algebraic decay in the spatial domain
and exponential decay in the frequency domain.

Since 
\begin{equation}\label{eq:0628-1}
  \mathcal{F}[u_H](k) = \frac{2^{1-h}|k|^{h-1 / 2} K_{h-\frac{1}{2}}(|k|)}{\Gamma(h)}.
\end{equation}
Here, $K_v(k)$ stands for the modified Bessel function of the second kind,
satisfying
\begin{equation}\label{eq:0628-2}
  K_v(k) \propto \sqrt{\frac{\pi}{2}} \frac{e^{-k}}{\sqrt{k}}\left(1+O\left(\frac{1}{k}\right)\right), \quad \text{as } |k| \rightarrow \infty.
\end{equation}
From \eqref{eq:0628-1} and \eqref{eq:0628-2} we know for $\mu + \frac{1}{2} \in \mathbb{N}$ and
$r \in \mathbb{N}, \, 0 \leqslant r \leqslant \mu + \frac{1}{2}$,
\begin{equation}\label{eq:1203-1}
  \partial^r \mathcal{F}[u_H](k) \lesssim_\mu e^{-c|k|}.
\end{equation}
By Corollary~\ref{cor:Laguerre-proj} we know if $\beta = a$, where $a$ is a constant, then
for $y = x^2,\,u(y) = u_H(x) = (1+y)^{-h}$, we have
\begin{equation}\label{eq:0628-3}
  \|u - \widehat{\Pi}_{N}^{(\mu,\beta,L)} u\|_{y^{\mu}} \lesssim_\mu N^{\frac{\mu+1}{2}-h}.
\end{equation}
If $\mu + \frac{1}{2} \notin \mathbb{N}$, by interpolation inequality \eqref{eq:0628-3} also holds.
$\beta = C h^2 (\ln N)^2 / N$ balances the spatial and frequency
error. Again, by Corollary~\ref{cor:Laguerre-proj}, we have
\begin{equation}\label{eq:0628-4}
  \|u - \widehat{\Pi}_{N}^{(\mu,\beta,L)} u\|_{y^{\mu}} \lesssim_\mu \left(N / \ln N\right)^{\mu+1-2h}.
\end{equation}
Next, we solve the model problem \eqref{eq:0624-18} by the Laguerre--Galerkin
method defined in \eqref{eq:0624-20} with $\gamma = 1$ and true solution
$u = (1+x)^{-h+1}-(1+x)^{-h} = x\cdot (1+x)^{-h}$.
Take $u = x/(1+x)^2$ as an example, with scaling factor
$\beta = 10$ and $\beta = 100/N$. Notice that the
scaling factor we choose here is slightly different from \eqref{eq:0628-4}.
If $\beta = C/N$, where $C$ is a large number, then by Corollary~\ref{cor:Laguerre-proj}
the frequency truncation error is negligible. In the pre-asymptotic regime, we have
\begin{equation}\label{eq:0628-5}
  \|u - \widehat{\Pi}_{N}^{(0,\beta,L)} u\| \lesssim N^{1-2h}.
\end{equation}
The $L^2$ error $\|u - u_{N,\beta}\|$ behaves like

\begin{figure}[H]
  \centering
  \includegraphics[width=0.65\linewidth]{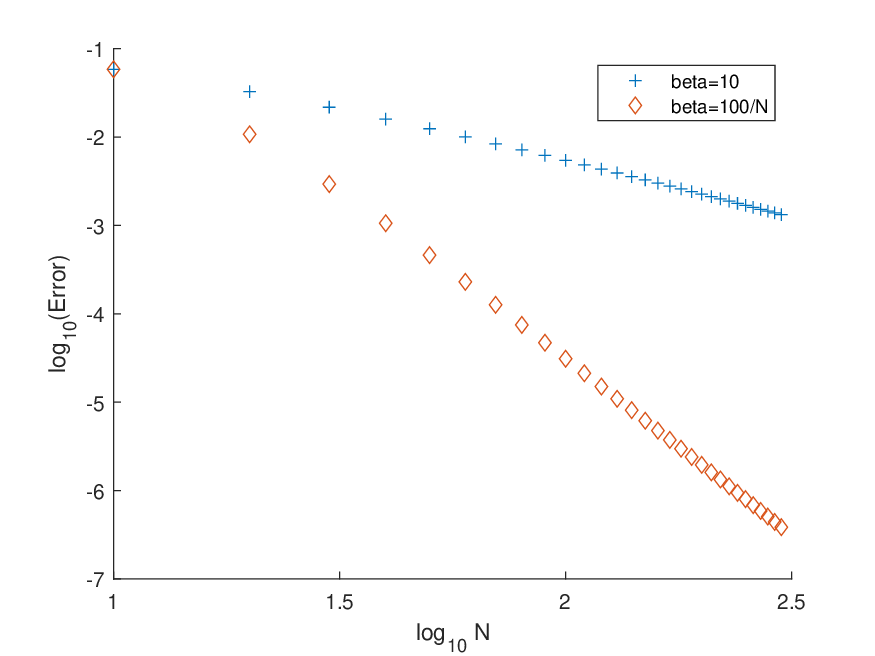}
  \caption{The $L^2$ error $\|u - u_{N,\beta}\|$ of the Laguerre--Galerkin method solving problem \eqref{eq:0624-18} with $u = x/(1+x)^2, \beta=10, \beta = 100/N$.  }
  \label{fig:x-2-beta-cmp}
\end{figure}

For $u = x/(1+x)^h$, taking different $h$, we list the convergence
orders of $\|u - u_{N,\beta}\|$ in Table~\ref{tab:conv-order}. Recall that we use \eqref{eq:0624-20}
to solve the model problem. Here we choose $\beta = 10$ and $\beta = 100/N$,
with $10 \leqslant N \leqslant 150$.

\begin{table}[h]
\caption{Convergence order of different scaling factor choices.}
\label{tab:conv-order}
\begin{tabular}{@{}lll@{}}
\toprule
    h & $\beta = 10$ & $\beta = 100/N$ \\ 
    \midrule
    1.6 & 0.732 & 2.76 \\
    1.8 & 0.914 & 3.14 \\ 
    2 & 1.10 & 3.52 \\ 
    2.2 & 1.28 & 3.90 \\ 
    2.4 & 1.46 & 4.29 \\
    2.6 & 1.64 & 4.67 \\
    \midrule
  \end{tabular}
\end{table}

It can be clearly seen from Table~\ref{tab:conv-order} that when $h$ increases 0.2, the convergence order for $\beta = 10$ increases about $0.2$, while for $\beta = 100/N$ it increases about $0.4$.
That is to say, a proper scaling doubles the convergence order. This fact holds for all functions having algebraic (exponential) decay in the spatial domain and exponential (algebraic) decay in the frequency domain.

\subsection{Why error in pre-asymptotic regime exhibits sub-geometric convergence}
\label{subsec: puzzling convergence}
When solving the model problem \eqref{eq:0624-18} with a true solution
$u = x(1+x)^{-h}$, \cite{shen2000stable} reported a sub-geometric convergence exp$(-c\sqrt{N})$.
This is considered puzzling since the classical error estimate
only predicts a convergence rate about $N^{-h}$.

However, through our insight, the sub-geometric convergence that occurs in the pre-asymptotic regime is
a natural result of Corollary~\ref{cor:Laguerre-proj}.
Let $u_H(x) = u(x^2) = x^2 (1+x^2)^{-h}$, notice that
\begin{displaymath}
  u_H(x) = \frac{1}{(1+x^2)^{h-1}} - \frac{1}{(1+x^2)^h},
\end{displaymath}
from \eqref{eq:1203-1} we know the frequency
truncation error is exp$(-c\sqrt{N})$. On the other hand, the spatial truncation error satisfies
\begin{equation}\label{eq:0704-2}
  \left\|u_H \cdot \mathbb{I}_{\left\{|x|>\frac{\sqrt{N}}{2 \sqrt{3}} \right\}}\right\|_{|x|^{2\mu}} \lesssim  N^{\frac{5}{4}+\frac{\mu}{2}-h}.
\end{equation}
Although asymptotically the spatial truncation error will be the dominant term, hence, the error has an order of about $N^{-h}$. In the pre-asymptotic regime, the frequency truncation error can be larger than the spatial truncation error, making the total error show a sub-geometric convergence order.

To make our argument more convincing, we will calculate the position where the error changes from a sub-geometric convergence $\exp(-c\sqrt{N})$ to an algebraic convergence about $N^{-h}$, then compare our results
with results of numerical experiments.

Let
\begin{equation}\label{eq:0704-3}
  E_s(M) = \left\|u_H \cdot \mathbb{I}_{\left\{|x|>M\right\}}\right\|_{|x|^{2\mu}},
\end{equation}
and
\begin{equation}\label{eq:0704-4}
  \begin{aligned}
    E_f(B) & =B^{-\mu+\frac{1}{2}} \left\| \mathcal{F}[u]\left(B \xi\right) \right\|_{H^{\lfloor\mu\rfloor}\left(\mathbb{R}\backslash [-1,1]\right)}^{{1}/{p}} \\
    &\quad \times \left\| \mathcal{F}[u]\left(B \xi\right) \right\|_{H^{\lceil\mu\rceil}\left(\mathbb{R}\backslash [-1,1]\right)}^{{1}/{q}}, \\
  \end{aligned}
\end{equation}
where $\frac{1}{p} \lfloor\mu\rfloor + \frac{1}{q} \lceil\mu\rceil = \mu$.
From Theorem \ref{thm:bigthm}, \eqref{eq:0704-3} and \eqref{eq:0704-4} we know that the spatial/frequency
truncation error can be recorded as $E_s\left(\frac{\sqrt{N}}{2\sqrt{3}}\right)$ and
$E_f\left(\frac{\sqrt{N}}{2\sqrt{3}}\right)$. Here, for computational convenience, 
we use the interpolation of integer-order Sobolev space norms as an approximation to fractional-order Sobolev space norms.
From the previous theoretical analysis, we know that the convergence rate changes from $\exp(-c\sqrt{N})$ to about $N^{-h}$ when spatial truncation error equals frequency truncation error, i.e.,
\begin{equation}\label{eq:0704-5}
  E_s\left(\frac{\sqrt{N}}{2\sqrt{3}}\right) = E_f\left(\frac{\sqrt{N}}{2\sqrt{3}}\right).
\end{equation}
Since the constant $\frac{1}{2\sqrt{3}}$ is not essential, we can only expect the transition
point $p$ where $E_f(p) = E_s(p)$ is proportional to $\sqrt{N}$.

Notice that for the Laguerre--Galerkin method \eqref{eq:0624-20} for problem with true solution
$u(x) = x(1+x)^{-h}$, we use Laguerre function $\widehat{\mathscr{L}_n}(x)$ to approximate $u(x)$, which is
equivalent to using $\widehat{H}_n^{(1/2)}(x)$ to approximate $u_H(x) = u(x^2) = x^2 (1+x^2)^{-h}$.
Hence for $E_s(M)$, $E_f(B)$ defined in \eqref{eq:0704-3} and \eqref{eq:0704-4}
with $\mu = \frac{1}{2}$ and $u_H(x) = x^2 (1+x^2)^{-h}$, we calculate the
transition point $p$ satisfying $E_f(p) = E_s(p)$.  Transition points $p$ 
with different $h$ are listed in Table~\ref{tab:transition-points}. Recall that according
to our analysis, $p$ should be proportional to $\sqrt{N}$.

\begin{table}[h]
\caption{Transition points $p$ satisfying $E_s(p)=E_f(p)$.}\label{tab:transition-points}
\begin{tabular}{@{}lllll@{}}
\hline 
$h$ & 3.0 & 3.5 & 4.0 & 4.5 \\ 
\hline
$p$ & 10.34 & 16.00 & 21.86 & 27.92 \\
\hline
\end{tabular}
\end{table}

Now, we solve the model problem \eqref{eq:0624-18} with $\gamma = 1$ and true solution $u(x) = x(1+x)^{-h}$
by the Laguerre--Galerkin method defined in \eqref{eq:0624-20} with $\beta = 1$. Taking different values of
$h$, the behavior of $\|u - u_{N,\beta}\|$ is presented in Figure~\ref{fig:pre-asymptotic}.

\begin{figure}[H]
  \centering
  \includegraphics[width=0.6\linewidth]{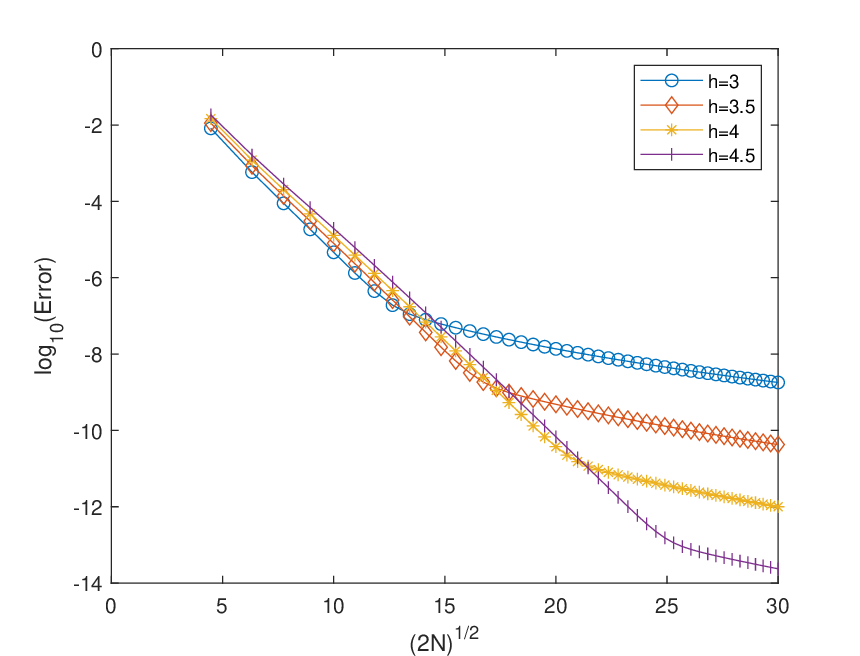}
  \caption{The $L^2$ error $\|u - u_{N,\beta}\|$ for problem \eqref{eq:0624-18} with $\gamma = 1$ and true solution $u = x/(1+x^2)^h$. Here $\beta = 1$.  }
  \label{fig:pre-asymptotic}
\end{figure}

From Table~\ref{tab:transition-points} and Figure~\ref{fig:pre-asymptotic}, we find transition
points $p$ perfectly match the position where the error transforms from sub-geometric to algebraic convergence. This proves the validity of our theory.

\subsection{Hermite versus Laguerre: different convergence behaviors}
\label{subsec: vs}
In Section 7.4.3 of~\cite{shen2011spectral}, the following test functions for 
the Laguerre and Hermite approximations are considered:
\begin{equation}\label{eq:shen-7.154}
  u(x)=\frac{\sin k x}{(1+x)^h} \text { for } x \in(0, \infty) \text { or } u(x)=\frac{\sin k x}{\left(1+x^2\right)^h} \text { for } x \in(-\infty, \infty) .
\end{equation}
Although these two functions are uniformly described as functions having
algebraic decay with oscillation at infinity in~\cite{shen2011spectral},
our theory shows that their approximation properties are completely different.
This once again proves that the approximation characterization in our framework
provides richer information than the classical theory.

Back to \eqref{eq:shen-7.154}, when using Hermite functions to approximate $u(x)=\frac{\sin k x}{\left(1+x^2\right)^h}$, it is easy to check that this function has algebraic decay in the spatial domain and exponential decay in the frequency domain. Hence, as in Section \ref{subsec: double conv}, a proper
scaling could double the convergence order.

However, when using Laguerre functions to approximate $u(x)=\frac{\sin k x}{(1+x)^h}$,
due to the equivalence, we need to analyze the error when using
generalized Hermite functions to approximate $u_H(x) = \frac{\sin k x^2}{(1+x^2)^h}$.

For spatial truncation error, when $\beta = 1$, we have
\begin{equation}\label{eq:0629-5.5-1}
  \left\|u_H \cdot \mathbb{I}_{\left\{|x|>\frac{\sqrt{N}}{2 \sqrt{3}}\right\}}\right\|_{|x|^{2\mu}} \lesssim N^{\frac{1}{4}+\frac{\mu}{2}-h}.
\end{equation}
For frequency truncation error, since $i^{s+r} \xi^s \partial^r \mathcal{F}[u_H](\xi) = \left(\mathcal{F}[\partial^s (u_Hx^r)](\xi)\right)$,
and $\partial^s (u_Hx^r)$ is in  $L^2$ for any $s < 2h-r-\frac{1}{2}$, we have
\begin{equation}\label{eq:0629-5.5-2}
  \begin{aligned}
    \left\|\partial^r \mathcal{F}\left[u_H\right](\xi) \cdot \mathbb{I}_{\left\{|\xi|>\frac{\sqrt{N}}{2 \sqrt{3}} \right\}}\right\| & \leqslant N^{-\frac{s}{2}} \left\|\xi^s \partial^r \mathcal{F}\left[u_H\right](\xi) \right\| \\
    & = N^{-\frac{s}{2}} \left\|\partial^s (u_H \cdot x^r)\right\|.
  \end{aligned}
\end{equation}
Hence for any $t > \frac{1}{4} + \frac{\mu}{2} -h$, the frequency truncation error defined in \eqref{eq:bigthm} satisfies
\begin{equation}\label{eq:0629-5.5-3}
  \begin{aligned}
    & \,\,\quad B^{-\mu+\frac{1}{2}} \left\| \mathcal{F}[u]\left(B \xi\right) \right\|_{H^{\mu}\left(\mathbb{R}\backslash [-1,1]\right)} \\
    & \leqslant B^{-\mu+\frac{1}{2}} \left\| \mathcal{F}[u]\left(B \xi\right) \right\|_{H^{\lfloor\mu\rfloor}\left(\mathbb{R}\backslash [-1,1]\right)}^{\frac{1}{p}} \times \left\| \mathcal{F}[u]\left(B \xi\right) \right\|_{H^{\lceil\mu\rceil}\left(\mathbb{R}\backslash [-1,1]\right)}^{\frac{1}{q}} \\
    & \lesssim \left(\sum_{r=0}^{\lfloor\mu\rfloor}\left\|\partial^r \mathcal{F}\left[u_H\right](\xi) \cdot \mathbb{I}_{\left\{|\xi|>\frac{\sqrt{N}}{2 \sqrt{3}} \right\}}\right\| \cdot( \sqrt{N})^{-(\lfloor\mu\rfloor-r)}\right)^{\frac{1}{p}} \\
    &\quad \times\left(\sum_{r=0}^{\lceil\mu\rceil}\left\|\partial^r \mathcal{F}\left[u_H\right](k) \cdot \mathbb{I}_{\left\{|\xi|>\frac{\sqrt{N}}{2 \sqrt{3}} \right\}}\right\| \cdot(\sqrt{N})^{-(\lceil\mu\rceil-r)}\right)^{\frac{1}{q}} \\
    & \lesssim_\mu N^t.
  \end{aligned}
\end{equation}
Here $p,q$ satisfy $\frac{1}{p}\lfloor \mu \rfloor + \frac{1}{q} \lceil \mu \rceil = 1$. Combining \eqref{eq:0629-5.5-1} and \eqref{eq:0629-5.5-3}, we know that
spatial and frequency truncation errors decay at almost the same algebraic
rate. Hence, the asymptotic convergence rate cannot be improved by scaling.
The best scaling factor to balance space and frequency is likely to be $O(1)$.
The following numerical experiment verifies our conclusion.

We solve the model problem \eqref{eq:0624-18} by the Laguerre--Galerkin
method defined in \eqref{eq:0624-20} with $\gamma = 1$, true solution
$u(x)=\sin 2 x \cdot (1+x)^{-\frac{7}{2}}$, the $L^2$ error $\|u - u_{N,\beta}\|$ is presented in Figure~\ref{fig:sin-x-h-beta-cmp}.

\begin{figure}[htbp]
  \centering
  \includegraphics[width=0.7\linewidth]{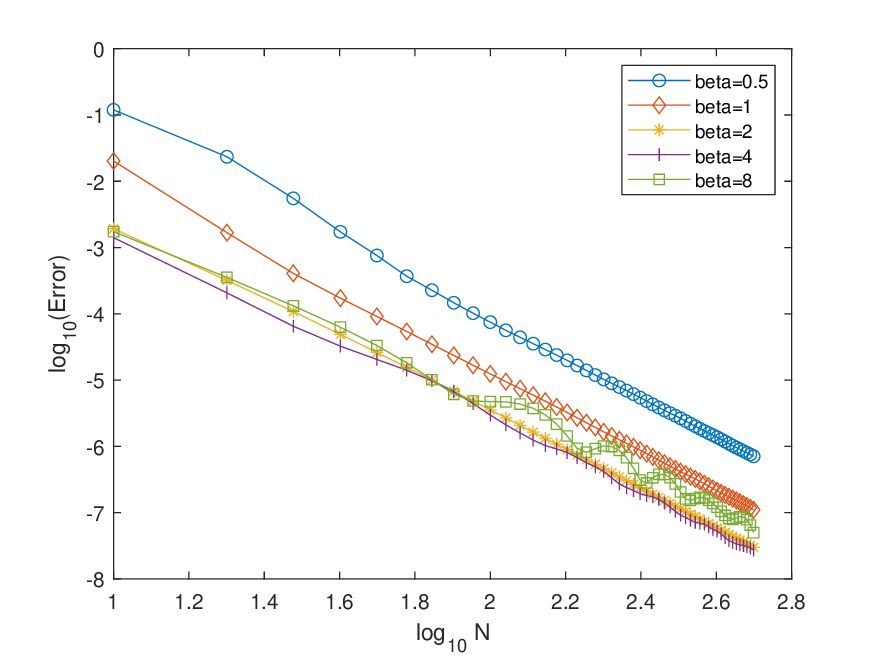}
  \caption{The $L^2$ error $\|u - u_{N,\beta}\|$ for problem \eqref{eq:0624-18} with $u = \sin 2x \cdot (1+x)^{-\frac{7}{2}}$ with different $\beta$.  }
  \label{fig:sin-x-h-beta-cmp}
\end{figure}

It can be seen from Figure~\ref{fig:sin-x-h-beta-cmp} that the best scaling factor is not
sensitive to the number of bases used, which is consistent with our error analysis.
Similar numerical results can be found in~\cite{huang2024improved}.

\subsection{Hermite versus Laguerre: Hermite or dual Laguerre?}
\label{subsec:1 Hermite vs 2 Laguerre}
This subsection aims to compare the performance of a single set of Hermite 
functions against two concatenated sets of Laguerre functions. 

It is intuitive that for functions exhibiting limited smoothness at zero (like $e^{-|x|}$) or disparate 
decay rates at $\pm \infty$, the use of two concatenated sets of Laguerre functions 
affords enhanced flexibility, which is conducive to improved approximation efficiency. 

However, in general, particularly with a scaling factor, there is no guarantee that 
two concatenated sets of Laguerre functions will outperform a single set of Hermite 
functions. For example, one might naturally expect that a function decaying at the rate of $e^{-x^2}$ is better approximated by Hermite functions ---which shares the same decay--- than by two concatenated sets of Laguerre functions.

Utilizing the error analysis framework developed in this paper, we will demonstrate that this intuition is incorrect. Specifically, we present the following theorem, which shows that for a wide class 
of functions decaying like $e^{-x^2}$, the approximation capability of two concatenated sets of Laguerre functions is no worse than---and may potentially surpass---that of a single set Hermite functions.

\begin{theorem}\label{thm:1201-1}
  For $g(x)$ bounded and analytic in the strip domain $-a \leqslant \operatorname{Im}(x) \leqslant a$,
  $u(x) = e^{-x^2}g(x)$, let scaling factor $\beta = N^{\frac{1}{6}}$, then approximating $u(x)$
  by a set of Hermite functions yields the following result
  \begin{equation}\label{eq:1201-1}
    \left\|u-\widehat{\Pi}_N^{\left(\mu, \beta\right)}u\right\|_{|x|^{2 \mu}} \lesssim_{\mu, u} e^{-c N^{\frac{2}{3}}}.
  \end{equation}
  Moreover, when $g(x)$ has singularities, the convergence rate estimated in
  \eqref{eq:1201-1} cannot be improved to $e^{-cN^\alpha}$ for $\alpha > \frac{2}{3}$ via Theorem~\ref{thm:bigthm}.
  
  In contrast, when approximating $u(x)$ using two sets of concatenated Laguerre functions, taking $\beta = N^{\frac{2}{3}}$ yields a convergence rate  no worse than that in \eqref{eq:1201-1}. This holds for both segments of $u(x)$. For instance, the following estimate holds for the part where $x \geqslant 0$
  \begin{equation}\label{eq:1201-2}
    \left\|u-\widehat{\Pi}_N^{\left(\mu, \beta, L\right)} u\right\|_{x^{\mu}} \lesssim_{\mu, u} e^{-c N^{\frac{2}{3}}}.
  \end{equation}
  Moreover, when $g(x)$ has singularities, the convergence rate estimated in
  \eqref{eq:1201-2} may still be further improved. For instance, if $g(x^2)$ remains a bounded analytic function in the strip region $-b \leqslant \operatorname{Im}(x) \leqslant b$, then by setting
  $\beta = N^{\frac{3}{5}}$, the following error estimate holds for the part of $u(x)$ where $x \geqslant 0$
  \begin{equation}\label{eq:1201-3}
    \left\|u-\widehat{\Pi}_N^{\left(\mu, \beta, L\right)}u\right\|_{x^{\mu}} \lesssim_{\mu, u} e^{-c N^{\frac{4}{5}}}.
  \end{equation}
  An analogous result applies to the part where $x < 0$.
\end{theorem}

The proof of this theorem can be found in Appendix~\ref{appendix:C}.

We verify the above results by comparing the errors from Hermite and Laguerre quadratures
for $\int_{-\infty}^{+\infty} \frac{e^{-x^2}}{1+16x^2} dx$.
Let $N$ be an odd integer. We apply Gauss--Laguerre--Radau quadrature with
$\frac{N+1}{2}$ nodes on each semi-axis to achieve a total of $N$ quadrature points, such that a direct comparison with Hermite quadrature under an identical node count is ensured.
Here, the quadrature rules correspond to the narrow-sense Hermite and Laguerre functions, specifically those denoted by $\widehat{H}_n(x)$ and $\widehat{\mathscr{L}}_n(x)$.

Since here $u(x) = \frac{e^{-x^2}}{1+16x^2}$, $g(x) = e^{x^2}u(x)$ satisfies
$g(x^2)$ is still a bounded and analytic function in the strip domain
$-b \leqslant \operatorname{Im}(x) \leqslant b$ for $b < \frac{1}{2\sqrt{2}}$,
it follows from Theorem~\ref{thm:1201-1} that the asymptotically optimal convergence rates of Hermite and Laguerre approximations are $e^{-cN^{\frac{2}{3}}}$ and $e^{-cN^{\frac{4}{5}}}$, respectively.

Similar to Section~\ref{subsec: puzzling convergence}, balancing the spatial and frequency errors 
outside the interpolation nodes leads to asymptotically optimal scaling factors of approximately
$\beta_{opt}^H = 2^{\frac{5}{6}}N^{\frac{1}{6}}$ for Hermite quadrature and
$\beta_{opt}^L = 2^{\frac{5}{6}}N^{\frac{3}{5}}$ for Laguerre quadrature.

The error results with optimal scaling factors $\beta_{opt}^H$ for Hermite and $\beta_{opt}^L$ for Laguerre quadrature are shown below, where $N$ is also taken from $11$ to $211$.
\begin{figure}[H]
\centering
\subfigure[Hermite quadrature errors with $N^{\frac{2}{3}}$ on the horizontal axis]{
  \begin{minipage}[t]{0.47\textwidth}
  \centering
  \includegraphics[width=\textwidth]{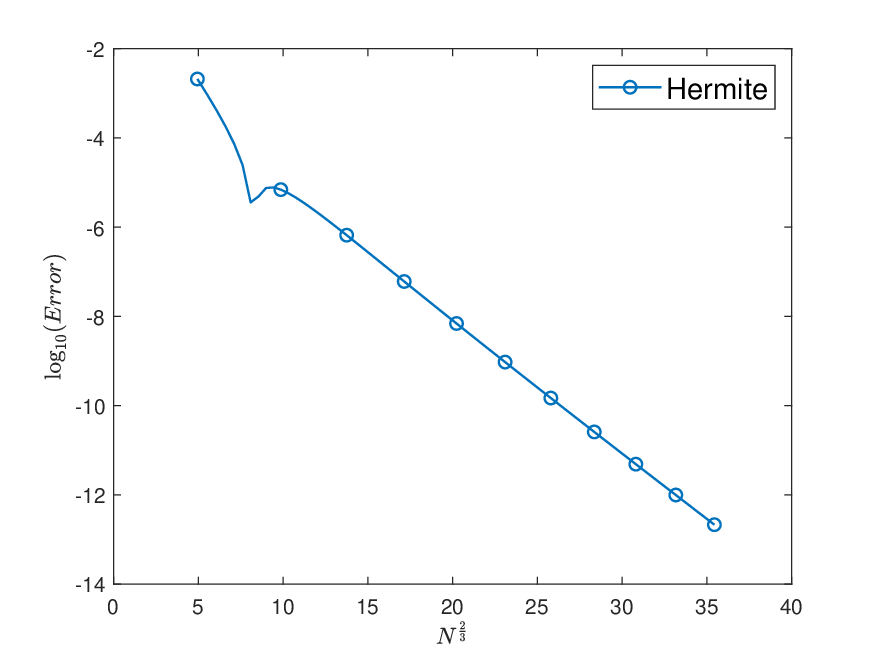}
  \label{fig:N23}
  \end{minipage}
}
\subfigure[Laguerre quadrature errors with $N^{\frac{4}{5}}$ on the horizontal axis]{
  \begin{minipage}[t]{0.47\textwidth}
  \centering
  \includegraphics[width=\textwidth]{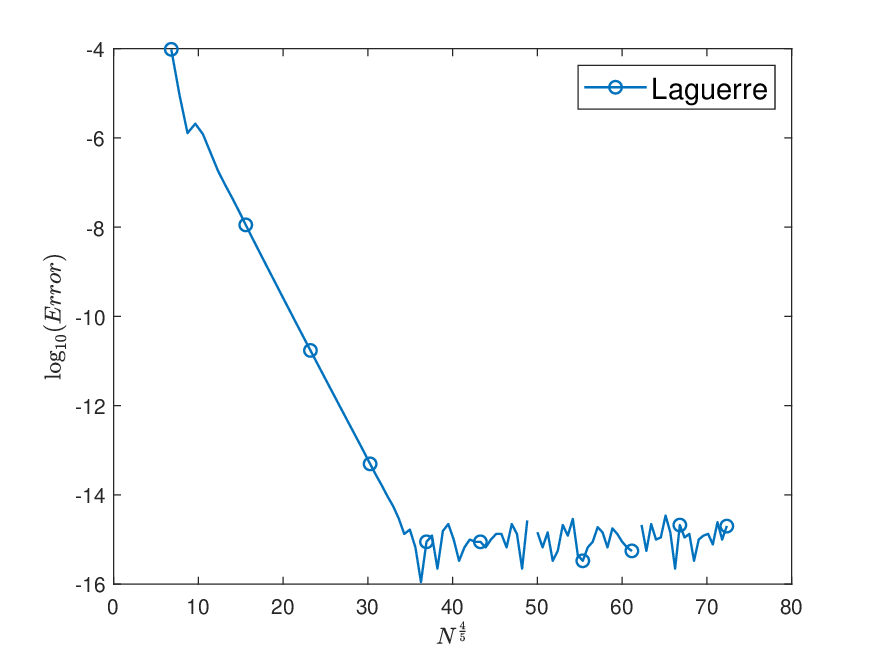}
  \label{fig:N45}
  \end{minipage}
}
\caption{Comparison of Hermite and Laguerre quadrature errors with optimal scaling factors}
\end{figure}

Figure~\ref{fig:N23} and Figure~\ref{fig:N45} verify the convergence orders 
of the optimal Hermite and Laguerre approximations, as given by \eqref{eq:1201-1} and \eqref{eq:1201-3}.

Similar to the above process, one can examine more examples where two optimally scaled, spliced Laguerre functions achieve comparable or superior approximation to a single optimally scaled Hermite function.

This pattern suggests that two spliced Laguerre functions 
under optimal scaling might be systematically superior. 
An intriguing question for future exploration is whether 
this can be proven as a general theorem.

\section{Conclusions}
\label{sec:conclusion}
In this paper, we propose a novel error analysis framework for the generalized Hermite and generalized  Laguerre approximation with scaling factors. 
It can be viewed as an analogue of the Nyquist-Shannon sampling theorem: 
we characterize the spatial and frequency bandwidth that can be 
captured by the generalized Hermite/Laguerre sampling points, and as long as the 
function satisfies the bandwidth restriction, it can be well approximated. Our characterization is 
more powerful than classical theory. As demonstrated by numerical experiments, it can not only systematically guide scaling, but also predict root-exponential and other more complex convergence behaviors
that cannot be characterized by classical theory. Using this characterization, we conducted an in-depth comparison of the behavior of Hermite and Laguerre approximations.
We find that functions with similar decay and oscillation characteristics can exhibit different convergence behaviors. Moreover, approximations using two concatenated sets of Laguerre functions may offer greater advantages compared to those using a single set of Hermite functions.

\backmatter

\bmhead{Acknowledgements}

We would like to acknowledge the financial support from the National Natural Science Foundation of China (grant
no. 12494543, 12171467) and
the Strategic Priority Research Program of the Chinese Academy of Sciences (grant no. XDA0480504).

\begin{appendices}
\section{Proof of \autoref{thm:0826-2}}
\label{appendix:A}
To prove Theorem~\ref{thm:0826-2}, we need the following lemma:
\begin{lemma}\label{lem:0826-1}
  Suppose that
  \begin{displaymath}
    u=\sum_{n=0}^{+\infty} c_n \widehat{H}_n^{(\mu)}, \quad \partial_x u=\sum_{n=0}^{+\infty} d_n \widehat{H}_n^{(\mu)},
  \end{displaymath}
  then
  \begin{equation}\label{eq:0826-1}
    \sum_n\left|d_n\right|^2 \lesssim_\mu \sum_n\left|c_n\right|^2 \cdot(n+1).
  \end{equation}
\end{lemma}
\begin{proof}
  By Lemma~\ref{lem:derivative recurrence}, when $2 \mid n$,
  \begin{equation}\label{eq:0826-2}
    \partial_x \widehat{H}_n^{(\mu)}=-\sqrt{\frac{n+1+2 \mu}{2}} \widehat{H}_{n+1}^{(\mu)}+\sqrt{\frac{n}{2}} \widehat{H}_{n-1}^{(\mu)},
  \end{equation}
  hence for $2 \nmid n$,
  \begin{equation}\label{eq:0826-3}
    d_n=-\sqrt{\frac{n+2 \mu}{2}} c_{n-1}+\sqrt{\frac{n+1}{2}} c_{n+1}.
  \end{equation}
  Then we have
  \begin{equation}\label{eq:0826-14}
    \sum_{2 \nmid n}\left|d_n\right|^2 \lesssim_\mu \sum_{2 \mid n}\left|c_n\right|^2 \cdot(n+1).
  \end{equation}
  When $2 \nmid n$, by Lemma~\ref{lem:derivative recurrence},
  \begin{equation}\label{eq:0826-5}
    \begin{aligned}
    \partial_x \widehat{H}_n^{(\mu)} & =-\sqrt{\frac{n+1}{2}} \widehat{H}_{n+1}^{(\mu)}+\sqrt{\frac{n+2 \mu}{2}} \widehat{H}_{n-1}^{(\mu)} \\
    & +2 \mu(-1)^{\frac{n-3}{2}} \times \sqrt{\frac{\left(\frac{n-1}{2}\right)!}{\Gamma\left(\frac{n}{2}+\mu+1\right)}} \\
    & \times \sum_{k=0}^{\frac{n-1}{2}}(-1)^k \sqrt{\frac{\Gamma\left(k+\mu+\frac{1}{2}\right)}{k!}} \widehat{H}_{2 k}^{(\mu)}.
    \end{aligned}
  \end{equation}
  Hence for $2 \mid n$,
  \begin{equation}\label{eq:0826-6}
    d_n=-\sqrt{\frac{n}{2}} c_{n-1}+\sqrt{\frac{n+1+2 \mu}{2}} c_{n+1}+\sum_{k=\frac{n}{2}}^{+\infty} w_{n, 2 k+1} c_{2 k+1},
  \end{equation}
  where $w_{n, 2k+1}$ is defined as
  \begin{equation}\label{eq:0826-7}
    w_{n, 2 k+1}=2 \mu(-1)^{k-1+\frac{n}{2}} \sqrt{\frac{k!}{\Gamma\left(k+\mu+\frac{3}{2}\right)}} \sqrt{\frac{\Gamma\left(\frac{n}{2}+\mu+\frac{1}{2}\right)}{\left(\frac{n}{2}\right)!}}.
  \end{equation}
  From \eqref{eq:0826-6} we know
  \begin{equation}\label{eq:0826-8}
    \left|d_n\right|^2 \lesssim_\mu \left|c_{n-1}\right|^2 \cdot n+\left|c_{n+1}\right|^2 \cdot(n+2)+\left|\sum_{k=\frac{n}{2}}^{+\infty} w_{n, 2 k+1} c_{2 k+1}\right|^2.
  \end{equation}
  Let $\alpha = 1 - \frac{1}{2}\left(\frac{1}{2}+\mu\right)$, since $\mu > -\frac{1}{2}$,
  $\alpha < 1$, by Cauchy--Schwarz inequality,
  \begin{equation}\label{eq:0826-9}
    \begin{aligned}
    \left|\sum_{k=\frac{n}{2}}^{+\infty} w_{n, 2 k+1} c_{2 k+1}\right|^2 &\leqslant {\left[\sum_{k=\frac{n}{2}}^{+\infty}\left|w_{n, 2 k+1}\right|^2 \cdot(2 k+1)^{-\alpha}\right]} \\
    &\times {\left[\sum_{k=\frac{n}{2}}^{+\infty}\left|c_{2 k+1}\right|^2 \cdot(2 k+1)^\alpha\right] }
    \end{aligned}
  \end{equation}
  From \eqref{eq:0826-7} we know
  \begin{equation}\label{eq:0826-10}
    \left|w_{n, 2 k+1}\right|^2 \lesssim_\mu(k+1)^{-\frac{1}{2}-\mu}(1+n)^{-\frac{1}{2}+\mu},
  \end{equation}
  hence
  \begin{equation}\label{eq:0826-11}
    \sum_{k=\frac{n}{2}}^{+\infty}\left|w_{n, 2 k+1}\right|^2(2 k+1)^{-\alpha} \lesssim_\mu(1+n)^{-\frac{1}{2}+\mu} \sum_{k=\frac{n}{2}}^{+\infty}(k+1)^{-\alpha-\left(\frac{1}{2}+\mu\right)}.
  \end{equation}
  Recall that $\alpha=1-\frac{1}{2}\left(\frac{1}{2}+\mu\right), \mu>-\frac{1}{2}$, we have
  $\alpha+\frac{1}{2}+\mu > 1$,
  \begin{equation}\label{eq:0826-12}
    \sum_{k=\frac{n}{2}}^{+\infty}(k+1)^{-\alpha-\left(\frac{1}{2}+\mu\right)} \lesssim_\mu(1+n)^{-\alpha+\frac{1}{2}-\mu}.
  \end{equation}
  Combining \eqref{eq:0826-12} with \eqref{eq:0826-11} yields
  \begin{displaymath}
    \sum_{k=\frac{n}{2}}^{+\infty}\left|w_{n, 2 k+1}\right|^2(2 k+1)^{-\alpha} \lesssim_\mu(1+n)^{-\alpha}.
  \end{displaymath}
  Putting the above formula back into \eqref{eq:0826-9}, we get
  \begin{equation}\label{eq:0826-12.5}
    \left|\sum_{k=\frac{n}{2}}^{+\infty} w_{n, 2 k+1} c_{2 k+1}\right|^2 \lesssim_\mu(1+n)^{-\alpha} \sum_{k=\frac{n}{2}}^{+\infty}\left|c_{2 k+1}\right|^2 \cdot(2 k+1)^\alpha.
  \end{equation}
  Combining \eqref{eq:0826-12.5} with \eqref{eq:0826-8}, we know that for $2 \mid n$,
  \begin{equation}\label{eq:0826-13}
    \begin{aligned}
    \left|d_n\right|^2 & \lesssim_\mu\left|c_{n-1}\right|^2 \cdot n+\left|c_{n+1}\right|^2 \cdot(n+2) \\
    & +(1+n)^{-\alpha} \sum_{k=\frac{n}{2}}^{+\infty}\left|c_{2 k+1}\right|^2 \cdot(2 k+1)^\alpha,
    \end{aligned}
  \end{equation}
  hence
  \begin{equation}\label{eq:0826-15}
    \sum_{2 \mid n}|d_n|^2 \lesssim_\mu \sum_{2 \nmid n}\left|c_n\right|^2 \cdot\left[(n+1)+n^\alpha \sum_{k=0}^{\frac{n-1}{2}}(1+2 k)^{-\alpha}\right].
  \end{equation}
  Recall that $\mu>-\frac{1}{2}, \alpha=1-\frac{1}{2}\left(\frac{1}{2}+\mu\right)<1$,
  hence
  \begin{equation}\label{eq:0826-16}
    \sum_{k=0}^{\frac{n-1}{2}}(1+2 k)^{-\alpha} \lesssim_\mu n^{1-\alpha}.
  \end{equation}
  Putting \eqref{eq:0826-16} back into \eqref{eq:0826-15} we have
  \begin{equation}\label{eq:0826-17}
    \sum_{2 \mid n}|d_n|^2 \lesssim_\mu \sum_{2 \nmid n}\left|c_n\right|^2 \cdot(n+1).
  \end{equation}
  Combining \eqref{eq:0826-14} with \eqref{eq:0826-17} completes our proof.
\end{proof}
Now we return to the proof of Theorem~\ref{thm:0826-2}.
\begin{proof}
  We proceed with the proof by using an induction argument.
  It is clear that \eqref{eq:0826-19} holds for $l = 0$.
  Assume that $\eqref{eq:0826-19}$ holds for $0,...,l-1$. Let
  \begin{equation}\label{eq:0826-20}
    \begin{aligned}
    \partial_x^{l-1}\left(u-\widehat{\Pi}_N^{(\mu, \beta)} u\right) & =\sum_{n=0}^{+\infty} c_n \widehat{H}_n^{(\mu)}(\beta x) \\
    \partial_x^l\left(u-\widehat{\Pi}_N^{(\mu, \beta)} u\right) & =\sum_{n=0}^{+\infty} d_n \widehat{H}_n^{(\mu)}(\beta x).
    \end{aligned}
  \end{equation}
  Let $y=\beta x, e=\partial_x^{l-1}\left(u-\widehat{\Pi}_N^{(\mu, \beta)} u\right)$, then
  \begin{equation}\label{eq:0826-21}
    \begin{aligned}
    e & =\sum_{n=0}^{+\infty} c_n \widehat{H}_n^{(\mu)}(y) \\
    \partial_y e & =\frac{1}{\beta} \partial_x e \\
    & =\sum_{n=0}^{+\infty} \frac{1}{\beta} d_n \widehat{H}_n^{(\mu)}(y).
    \end{aligned}
  \end{equation}
  Using Lemma~\ref{lem:0826-1} yields
  \begin{equation}\label{eq:0826-22}
    \sum_{n=0}^{+\infty}\left|d_n\right|^2 \lesssim_\mu \sum_{n=0}^{+\infty}\left|c_n\right|^2 \cdot(n+1) \cdot \beta^2.
  \end{equation}
  Notice that
  \begin{equation}\label{eq:0826-23}
    \begin{aligned}
      \sum_{n=0}^{+\infty}\left|c_n\right|^2 \cdot(n+1) &=\sum_{n=0}^{N+l}\left|c_n\right|^2 \cdot(n+1)+\sum_{n>N+l}\left|c_n\right|^2 \cdot(n+1) \\
      &\triangleq P_1 + P_2.
    \end{aligned}
  \end{equation}
  For $P_1$,
  \begin{equation}\label{eq:0826-24}
    \sum_{n=0}^{N+l}\left|c_n\right|^2 \cdot(n+1) \lesssim_l N \cdot \sum_{n=0}^{+\infty}\left|c_n\right|^2.
  \end{equation}
  Since
  \begin{equation}\label{eq:0826-25}
    \sum_{n=0}^{+\infty}\left|c_n\right|^2=\beta \cdot\left\|\partial_x^{l-1}\left(u-\widehat{\Pi}_N^{(\mu, \beta)} u\right)\right\|_{|x|^{2 \mu}}^2,
  \end{equation}
  by induction hypothesis, we have
  \begin{equation}\label{eq:0826-26}
    \sum_{n=0}^{+\infty}\left|c_n\right|^2 \lesssim_{\mu, l} \left(\int_N^{+\infty}|f(x)|^2 \cdot x^{l-1} d x\right) \cdot \beta^{2 l-1}.
  \end{equation}
  Combining with \eqref{eq:0826-24} yields
  \begin{equation}\label{eq:0826-27}
    \sum_{n=0}^{N+l}\left|c_n\right|^2 \cdot(n+1) \lesssim_{\mu, l} N \cdot \int_N^{+\infty}|f(x)|^2 \cdot x^{l-1} d x \cdot \beta^{2 l-1}.
  \end{equation}
  For $P_2$ defined in \eqref{eq:0826-23}, we have
  \begin{equation}\label{eq:0826-28}
    \sum_{n>N+l}\left|c_n\right|^2 \cdot(n+1)=(N+l+1) \sum_{n > N+l}\left|c_n\right|^2+\sum_{n \geqslant N+l} \sum_{m>n}\left|c_m\right|^2.
  \end{equation}
  By Lemma~\ref{lem:derivative recurrence}, we know that
  \begin{equation}\label{eq:0826-28.5}
    \partial_x^{l-1} \left(\widehat{\Pi}_{n-l+1}^{(\mu, \beta)} u\right) = \sum_{k = 0}^{n} a_k \widehat{H}_k^{(\mu)}(\beta x),
  \end{equation}
  then combine with \eqref{eq:0826-20} we get
  \begin{equation}\label{eq:0826-29}
    \widehat{\Pi}_n^{(\mu, \beta)}\left(\partial_x^{l-1}\left(u-\widehat{\Pi}_{n-l+1}^{(\mu, \beta)} u\right)\right)=\sum_{m>n} c_m \widehat{H}_m^{(\mu)}(\beta x).
  \end{equation}
  By induction hypothesis,
  \begin{equation}\label{eq:0826-30}
    \begin{aligned}
    \sum_{m>n}\left|c_m\right|^2 & =\beta \cdot\left\|\widehat{\Pi}_n^{(\mu, \beta)}\left(\partial_x^{l-1}\left(u-\widehat{\Pi}_{n-l+1}^{(\mu, \beta)} u\right)\right)\right\|_{|x|^{2 \mu}}^2 \\
    & \leqslant \beta \cdot\left\|\partial_x^{l-1}\left(u-\widehat{\Pi}_{n-l+1}^{(\mu, \beta)} u\right)\right\|_{|x|^{2 \mu}}^2 \\
    & \lesssim{ }_{\mu, l} \int_{n-l+1}^{+\infty}|f(x)|^2 \cdot x^{l-1} d x \cdot \beta^{2 l-1}.
    \end{aligned}
  \end{equation}
  Putting \eqref{eq:0826-30} back into \eqref{eq:0826-28} yields
  \begin{equation}\label{eq:0826-31}
    \begin{aligned}
    \sum_{n \geqslant N+l} \sum_{m>n}\left|c_m\right|^2 & \lesssim_{\mu, l} \sum_{n \geqslant N+l} \int_{n-l+1}^{+\infty}|f(x)|^2 \cdot x^{l-1} d x \cdot \beta^{2 l-1} \\
    & \lesssim_{\mu, l} \int_N^{+\infty} \int_y^{+\infty}|f(x)|^2 \cdot x^{l-1} d x \cdot \beta^{2 l-1} d y \\
    & =\beta^{2 l-1} \cdot\left(\int_N^{+\infty} |f(x)|^2 \cdot x^l d x-N \int_N^{+\infty}|f(x)|^2 \cdot x^{l-1} d x\right).
    \end{aligned}
  \end{equation}
  Similar to \eqref{eq:0826-27}, we have
  \begin{equation}\label{eq:0826-32}
    (N+l+1) \sum_{n>N+l}\left|c_n\right|^2 \lesssim_{\mu, l} N \cdot \int_N^{+\infty}|f(x)|^2 \cdot x^{l-1} d x \cdot \beta^{2 l-1}.
  \end{equation}
  Putting \eqref{eq:0826-31}, \eqref{eq:0826-32} back into \eqref{eq:0826-28} yields
  \begin{equation}\label{eq:0826-33}
    \sum_{n>N+l}\left|c_n\right|^2 \cdot(n+1) \lesssim_{\mu, l} \beta^{2 l-1} \int_N^{+\infty}|f(x)|^2 \cdot x^l d x.
  \end{equation}
  Combining \eqref{eq:0826-27} with \eqref{eq:0826-33}, we have
  \begin{equation}\label{eq:0826-34}
    \sum_{n=0}^{+\infty}\left|c_n\right|^2 \cdot(n+1) \lesssim_{\mu, l} \beta^{2 l-1} \int_N^{+\infty}|f(x)|^2 \cdot x^l d x.
  \end{equation}
  Putting \eqref{eq:0826-34} back into \eqref{eq:0826-22}, we get
  \begin{equation}\label{eq:0826-35}
    \begin{aligned}
    \left\|\partial_x^l\left(u-\widehat{\Pi}_N^{(\mu, \beta)} u\right)\right\|_{|x|^{2 \mu}}^2 & =\sum_{n=0}^{+\infty}|d_n|^2 / \beta \\
    & \lesssim_{\mu, l} \sum_{n=0}^{+\infty}\left|c_n\right|^2 \cdot(n+1) \cdot \beta \\
    & \lesssim_{\mu, l} \beta^{2 l} \int_N^{+\infty}|f(x)|^2 \cdot x^l d x.
    \end{aligned}
  \end{equation}
  This completes our proof.
\end{proof}

\section{Proof of \autoref{thm:0810-1}}
\label{appendix:B}
To prove Theorem~\ref{thm:0810-1}, we first give two lemmas.
\begin{lemma}\label{lem:Guo2006-weights}
  Let $\omega_{\alpha,\beta}(x)$ denote $x^\alpha e^{-\beta x}$.
  $A \sim_\alpha B$ means $(A / B)^{ \pm 1} \leqslant C_\alpha$, where $C_\alpha$ is a
  constant depends on $\alpha$. Then we have
  \begin{equation}\label{eq:Guo2006-2.19}
    \begin{aligned}
    \omega_{G, N, j}^{(\alpha, \beta)} & =\frac{1}{\beta^{\alpha+1}} \omega_{G, N, j}^{(\alpha)} \sim_\alpha \frac{1}{\beta^{\alpha+1}} \omega_\alpha\left(\xi_{G, N, j}^{(\alpha)}\right)\left(\xi_{G, N, j+1}^{(\alpha)}-\xi_{G, N, j}^{(\alpha)}\right) \\
    & =\omega_{\alpha, \beta}\left(\xi_{G, N, j}^{(\alpha, \beta)}\right)\left(\xi_{G, N, j}^{(\alpha, \beta)}-\xi_{G, N, j-1}^{(\alpha, \beta)}\right), \quad 0 \leqslant j \leqslant N.
    \end{aligned}
  \end{equation}
  As for $\omega_{R,N,j}^{(\alpha,\beta)}$, we have
  \begin{equation}\label{eq:Guo2006-2.12}
    \omega_{R, N, j}^{(\alpha, \beta)}=\frac{1}{\beta^{\alpha+1}} \omega_{R, N, j}^{(\alpha)}= \begin{cases}\frac{(\alpha+1) \Gamma^2(\alpha+1) \Gamma(N+1)}{\beta^{\alpha+1} \Gamma(N+\alpha+2)}, & j=0, \\ \frac{\Gamma(N+\alpha+1)}{\beta^\alpha \Gamma(N+2)} \frac{1}{\mathcal{L}_{N+1}^{(\alpha, \beta)}\left(\xi_{R, N, j}^{(\alpha, \beta)}\right) \partial_x \mathcal{L}_N^{(\alpha, \beta)}\left(\xi_{R, N, j}^{(\alpha, \beta)}\right)}, & 1 \leqslant j \leqslant N,\end{cases}
  \end{equation}
  besides,
  \begin{equation}\label{eq:Guo2006-2.21}
    \begin{aligned}
    \omega_{R, N, j}^{(\alpha, \beta)} & =\left(\xi_{R, N, j}^{(\alpha, \beta)}\right)^{-1} \omega_{G, N-1, j-1}^{(\alpha+1, \beta)} \\
    & \sim_\alpha \left(\xi_{R, N, j}^{(\alpha, \beta)}\right)^{-1} \omega_{\alpha+1, \beta}\left(\xi_{G, N-1, j-1}^{(\alpha+1, \beta)}\right)\left(\xi_{G, N-1, j-1}^{(\alpha+1, \beta)}-\xi_{G, N-1, j-2}^{(\alpha+1, \beta)}\right) \\
    & =\omega_{\alpha, \beta}\left(\xi_{R, N, j}^{(\alpha, \beta)}\right)\left(\xi_{R, N, j}^{(\alpha, \beta)}-\xi_{R, N, j-1}^{(\alpha, \beta)}\right), \quad 1 \leqslant j \leqslant N .
    \end{aligned}
  \end{equation}
\end{lemma}
\eqref{eq:Guo2006-2.19}, \eqref{eq:Guo2006-2.12} and \eqref{eq:Guo2006-2.21} are
(2.19), (2.12) and (2.21) of \cite{ben2006generalized}.

\begin{lemma}\label{lem:0810-3}
  For $\widehat{\omega}_{2N+1,j}^{(\mu,\beta)}$ defined in \eqref{Hermite weights def},
  we have
  \begin{equation}\label{eq:0810-7}
    \widehat{\omega}_{2 N+1, j}^{(\mu,\beta)} \sim_\mu\left|x^{(\mu,\beta)}_{2N+1,j}\right|^{2 \mu-1}\left(\left|x^{(\mu,\beta)}_{2N+1,j}\right|^2-\left|x^{(\mu,\beta)}_{2N+1,j'}\right|^2\right),
  \end{equation}
  where
  \begin{equation}\label{eq:0810-7.5}
    j^{\prime}= \begin{cases}j+1 & 0 \leqslant j \leqslant N-1, \\ N+\frac{1}{2} & j=N, N+1, \\ j-1 & N+2 \leqslant j \leqslant 2 N+1.\end{cases}
  \end{equation}
  Here we set $x^{(\mu,\beta)}_{2N+1,N+\frac{1}{2}} = 0$.

  As for $\widehat{\omega}^{(\mu,\beta)}_{2N,j}$, we have
  \begin{equation}\label{eq:0810-8}
    \widehat{\omega}_{2 N, N}^{(\mu,\beta)} \sim_\mu \left(\beta \sqrt{N}\right)^{-2\mu-1},
  \end{equation}
  and for $j \neq N$,
  \begin{equation}\label{eq:0810-8.5}
    \widehat{\omega}_{2 N, j}^{(\mu,\beta)} \sim_\mu\left|x^{(\mu,\beta)}_{2N,j}\right|^{2 \mu-1}\left(\left|x^{(\mu,\beta)}_{2N,j}\right|^2-\left|x^{(\mu,\beta)}_{2N,j'}\right|^2\right),
  \end{equation}
  where $j'$ satisfies
  \begin{equation}\label{eq:0810-8.75}
    j^{\prime}= \begin{cases}j+1 & 0 \leqslant j \leqslant N-1, \\ j-1 & N+1 \leqslant j \leqslant 2 N.\end{cases}
  \end{equation}
\end{lemma}
\begin{proof}
  We only prove \eqref{eq:0810-8} and \eqref{eq:0810-8.5}. The
  proof of \eqref{eq:0810-7} is similar.

  Let $v(y) \in \widehat{P}_N^{\beta^2, L},\, y = x^2, u(x) = v(y)$, then
  $u(x) \in \widehat{P}_{2N}^{\beta, H}$. According to the integration by substitution formula,
  \begin{equation}\label{eq:0831-5}
    \int_0^{+\infty} v(y) y^{\mu-\frac{1}{2}} d y=\int_{-\infty}^{+\infty} u(x)|x|^{2 \mu} d x.
  \end{equation}
  Combining \eqref{eq:0831-5} with \eqref{Laguerre weights def} and
  \eqref{Hermite weights def} yields
  \begin{equation}\label{eq:0831-6}
    \sum_{j=0}^{2 N} u\left(x_{2 N, j}^{(\mu, \beta)}\right) \widehat{\omega}_{2 N, j}^{(\mu, \beta)}=\sum_{j=0}^N v\left(\xi_{R, N, j}^{\left(\mu-\frac{1}{2}, \beta^2\right)}\right) \widehat{\omega}_{R,N,j}^{\left(\mu-\frac{1}{2},\beta^2\right)}.
  \end{equation}
  For $0 \leqslant i \leqslant 2N$, define $u_i(x) \in \widehat{P}_{2 N}^{\beta, H}$ such that
  \begin{equation}\label{eq:0903-7}
    u_i\left(x_{2 N, j}^{(\mu, \beta)}\right)=\delta_{i j},
  \end{equation}
  then for $1 \leqslant j \leqslant N$, we know that $u_{N-j}(x) = u_{N+j}(-x)$, combining
  with \eqref{Hermite weights def} yields
  \begin{equation}\label{eq:0903-8}
    \begin{aligned}
    \widehat{\omega}_{2 N, N-j}^{\left(\mu, \beta\right)} & =\int_{-\infty}^{+\infty} u_{N-j}(x)|x|^{2 \mu} d x \\
    & =\int_{-\infty}^{+\infty} u_{N+j}(x)|x|^{2 \mu} d x \\
    & =\widehat{\omega}_{2 N, N+j}^{\left(\mu, \beta\right)}.
    \end{aligned}
  \end{equation}
  Notice that
  $x_{2N,j}^{\left(\mu,\beta\right)},\, 0\leqslant j \leqslant 2N$ are zeros of
  $H_{2N+1}^{(\mu)}(\beta x)$, $\xi_{R,N,j}^{\left(\mu-\frac{1}{2},\beta^2\right)}$ are
  zeros of $y\mathscr{L}_N^{\left(\mu+\frac{1}{2}\right)}(\beta^2 y)$, Combining the 
  conversion relationship \eqref{eq:relaion-Laguerre-Hermite} between Hermite and Laguerre polynomials yields
  \begin{equation}\label{eq:0831-4}
    \xi_{R, N, j}^{\left(\mu-\frac{1}{2}, \beta^2\right)}=\left|x_{2 N, N \pm j}^{(\mu, \beta)}\right|^2.
  \end{equation}
  Let $u(x) = u_N(x)$, then $v(y)=v(x^2)=u(x)$ satisfying
  \begin{displaymath}
    v\left(\xi_{R,N,j}^{\left(\mu-\frac{1}{2},\beta^2\right)}\right) = \delta_{0j}.
  \end{displaymath}
  Substituting $u(x)$ and $v(y)$ back into \eqref{eq:0831-6} yields
  \begin{equation}\label{eq:0903-9}
    \widehat{\omega}_{2 N, N}^{(\mu, \beta)}=\widehat{\omega}_{R, N, 0}^{\left(\mu-\frac{1}{2}, \beta^2\right)}.
  \end{equation}
  A further substitution of $u(x)=u_{N-j}(x)+u_{N+j}(x), \,v(y)=v\left(x^2\right)=u(x),\,1\leqslant j \leqslant N$
  into \eqref{eq:0831-6} yields
  \begin{equation}\label{eq:0903-10}
    \widehat{\omega}_{2 N, N-j}^{(\mu, \beta)}+\widehat{\omega}_{2 N, N+j}^{(\mu, \beta)}=\widehat{\omega}_{R, N, j}^{\left(\mu-\frac{1}{2}, \beta^2\right)}.
  \end{equation}
  Combining with \eqref{eq:0903-8} we know
  \begin{equation}\label{eq:0903-11}
    \widehat{\omega}_{2 N, N-j}^{(\mu, \beta)}=\widehat{\omega}_{2 N, N+j}^{(\mu, \beta)}=\frac{1}{2} \widehat{\omega}_{R, N, j}^{\left(\mu-\frac{1}{2}, \beta^2\right)}.
  \end{equation}
  From \eqref{eq:0903-9}, \eqref{eq:0903-12} and \eqref{eq:Guo2006-2.12} we have
  \begin{equation}\label{eq:0903-13}
    \begin{aligned}
    \widehat{\omega}_{2 N, N}^{(\mu, \beta)} & =\widehat{\omega}_{R, N, 0}^{\left(\mu-\frac{1}{2}, \beta^2\right)} \\
    & =\omega_{R, N, 0}^{\left(\mu-\frac{1}{2}, \beta^2\right)} \\
    & =\frac{\left(\mu+\frac{1}{2}\right) \Gamma^2\left(\mu+\frac{1}{2}\right) \Gamma(N+1)}{\beta^{2 \mu+1} \Gamma\left(N+\mu+\frac{3}{2}\right)} \\
    & \sim_\mu(\beta \sqrt{N})^{-2\mu-1}.
    \end{aligned}
  \end{equation}
  This completes the proof of \eqref{eq:0810-8}.
  Moreover, from \eqref{eq:0903-11}, \eqref{eq:0903-12} and \eqref{eq:Guo2006-2.21},
  it follows that
  \begin{equation}
    \begin{aligned}
    \widehat{\omega}_{2 N, N \pm j}^{(\mu, \beta)} & =\frac{1}{2} \widehat{\omega}_{R, N, j}^{\left(\mu-\frac{1}{2}, \beta^2\right)} \\
    & =\frac{1}{2} e^{\beta^2 \xi_{R, N, j}^{\left(\mu-\frac{1}{2}, \beta^2\right)}} \omega_{R, N, j}^{\left(\mu-\frac{1}{2}, \beta^2\right)} \\
    & \sim_\mu\left|\xi_{R, N, j}^{\left(\mu-\frac{1}{2}, \beta^2\right)}\right|^{\mu-\frac{1}{2}}\left(\xi_{R, N, j}^{\left(\mu-\frac{1}{2}, \beta^2\right)}-\xi_{R, N, j-1}^{\left(\mu-\frac{1}{2}, \beta^2\right)}\right).
    \end{aligned}
  \end{equation}
  Combining with \eqref{eq:0831-4} yields
  \begin{equation}\label{eq:0903-15}
    \widehat{\omega}_{2 N, N \pm j}^{(\mu, \beta)} \sim_\mu\left|x_{2 N, N \pm j}^{(\mu, \beta)}\right|^{2 \mu-1}\left(\left|x_{2 N, N \pm j}^{(\mu, \beta)}\right|^2-\left|x_{2 N, N \pm j \mp 1}^{(\mu, \beta)}\right|^2\right).
  \end{equation}
  This completes the proof of \eqref{eq:0810-8.5}.
\end{proof}
Now we return to the proof of \autoref{thm:0810-1}.
\begin{proof}
  Notice that
  \begin{equation}\label{eq:0810-14}
    \begin{aligned}
    \left\|u-\widehat{I}_{2 N}^{(\mu, \beta)} u\right\|_\omega & \leqslant\left\|u-\widehat{\Pi}_{2 N}^{(\mu, \beta)} u\right\|_\omega+\left\|\widehat{I}_{2 N}^{(\mu, \beta)}\left(u-\widehat{\Pi}_{2 N}^{(\mu, \beta)} u\right)\right\|_\omega \\
    & \triangleq E_1+E_2.
    \end{aligned}
  \end{equation}
  Let $e:=u-\widehat{\Pi}_{2 N}^{(\mu, \beta)} u$, thanks to the
  exactness of quadrature \eqref{Hermite weights def},
  \begin{equation}\label{eq:0810-15}
    \begin{aligned}
    E_2^2 & =\int_{-\infty}^{\infty}\left|\widehat{I}_{2 N}^{(\mu, \beta)} e\right|^2|x|^{2 \mu} dx\\
    & =\sum_{j=0}^{2N} \widehat{\omega}_{2 N, j}^{(\mu, \beta)} \left|e\left(x_{2N,j}^{(\mu,\beta)}\right)\right|^2.
    \end{aligned}
  \end{equation}
  From Lemma~\ref{lem:0810-3} we have
  \begin{equation}\label{eq:0810-16}
    \begin{aligned}
      E_{21} & = \widehat{\omega}_{2 N, N}^{(\mu, \beta)}|e(0)|^2 \\
      & \lesssim_\mu \beta^{-2 \mu-1} N^{-\mu-\frac{1}{2}}|e(0)|^2.
    \end{aligned}
  \end{equation}
  Applying Lemma~\ref{lem:0810-3} again, we obtain
  \begin{equation}\label{eq:0810-17}
    \begin{aligned}
    E_{22} & = \sum_{j=1}^N \widehat{\omega}_{2 N, N+j}^{(\mu, \beta)}\left|e\left(x_{2N,N+j}^{(\mu,\beta)}\right)\right|^2 \\
    & \lesssim_\mu \sum_{j=1}^N \left(x_{2N,N+j}^{(\mu,\beta)}\right)^{2 \mu-1}\left(\left(x_{2N,N+j}^{(\mu,\beta)}\right)^2-\left(x_{2N,N+j-1}^{(\mu,\beta)}\right)^2\right) \\
    & \times \left|e\left(x_{2N,N+j}^{(\mu,\beta)}\right)\right|^2 \\
    & \leqslant\sup \left\{\left(x_{2N,N+j}^{(\mu,\beta)}\right)^{2 \mu-1}\left(x_{2N,N+j}^{(\mu,\beta)}+x_{2N,N+j-1}^{(\mu,\beta)}\right)\left|e\left(x_{2N,N+j}^{(\mu,\beta)}\right)\right|^2\right\} \\
    & \times \sum_{j=1}^N\left(x_{2N,N+j}^{(\mu,\beta)}-x_{2N,N+j-1}^{(\mu,\beta)}\right).
    \end{aligned}
  \end{equation}
  From (7.44) of \cite{shen2011spectral} and \eqref{eq:relaion-Laguerre-Hermite} we know
  \begin{equation}\label{eq:0810-18}
    \sum_{j=1}^N\left(x_{2N,N+j}^{(\mu,\beta)}-x_{2N,N+j-1}^{(\mu,\beta)}\right)=x_{2N,2N}^{(\mu,\beta)} \lesssim_\mu \sqrt{N}/\beta.
  \end{equation}
  From the Sobolev inequality (see lemma B.4 of \cite{shen2011spectral})
  \begin{equation}\label{eq:0810-19}
    \max _{x \in[a, b]}|u(x)| \leqslant \frac{1}{\sqrt{b-a}}\|u\|_{L^2[a, b]}+\sqrt{b-a}\left\|u^{\prime}\right\|_{L^2[a, b]}
  \end{equation}
  we know
  \begin{equation}\label{eq:0810-20}
    \left|e\left(x_{2N,N+j}^{(\mu,\beta)}\right)\right|^2 \lesssim \int_{x_{2N,N+j}^{(\mu,\beta)}}^{x_{2N,N+j}^{(\mu,\beta)}+\left(\beta\sqrt{N}\right)^{-1}} \beta\sqrt{N} e^2 + \left(\beta\sqrt{N}\right)^{-1} \left(e^{\prime}\right)^2 d x.
  \end{equation}
  Hence
  \begin{equation}\label{eq:0810-21}
    \begin{aligned}
      & \left(x_{2N,N+j}^{(\mu,\beta)}\right)^{2 \mu-1}\left(x_{2N,N+j}^{(\mu,\beta)}+x_{2N,N+j-1}^{(\mu,\beta)}\right)\left|e\left(x_{2N,N+j}^{(\mu,\beta)}\right)\right|^2 \\
      \lesssim & \int_{x_{2N,N+j}^{(\mu,\beta)}}^{x_{2N,N+j}^{(\mu,\beta)}+\left(\beta\sqrt{N}\right)^{-1}} \beta\sqrt{N} e^2 |x|^{2\mu} + \left(\beta\sqrt{N}\right)^{-1} \left(e^{\prime}\right)^2 |x|^{2\mu} d x \\
      \lesssim & \beta \sqrt{N} \left\|e\right\|_\omega^2 + \left(\beta\sqrt{N}\right)^{-1} \left\|e^{\prime}\right\|_\omega^2.
    \end{aligned}
  \end{equation}
  Substituting \eqref{eq:0810-18} and \eqref{eq:0810-21} back into
  \eqref{eq:0810-17} yields
  \begin{equation}\label{eq:0810-22}
    \begin{aligned}
      E_{22} & = \sum_{j=1}^N \widehat{\omega}_{2 N, N+j}^{(\mu, \beta)}\left|e\left(x_{2N,N+j}^{(\mu,\beta)}\right)\right|^2 \\
      & \lesssim_\mu N \left\|e\right\|_\omega^2 + \beta^{-2} \left\|e^{\prime}\right\|_\omega^2.
    \end{aligned}
  \end{equation}
  Similarly,
  \begin{equation}\label{eq:0810-22.5}
    \begin{aligned}
      E_{23} & = \sum_{j=1}^N \widehat{\omega}_{2 N, N-j}^{(\mu, \beta)}\left|e\left(x_{2N,N+j}^{(\mu,\beta)}\right)\right|^2 \\
      & \lesssim_\mu N \left\|e\right\|_\omega^2 + \beta^{-2} \left\|e^{\prime}\right\|_\omega^2.
    \end{aligned}
  \end{equation}
  Hence for $E_2$ defined in \eqref{eq:0810-14}, combining
  \eqref{eq:0810-16}, \eqref{eq:0810-22} with \eqref{eq:0810-22.5}
  yields
  \begin{equation}\label{eq:0810-23}
    \begin{aligned}
      E_2^2 &= E_{21} + E_{22} + E_{23} \\
      & \lesssim_\mu (\beta \sqrt{N})^{-2 \mu-1}|e(0)|^2 + N \left\|e\right\|_\omega^2 + \beta^{-2} \left\|e^{\prime}\right\|_\omega^2.
    \end{aligned}
  \end{equation}
  From \eqref{eq:0810-14} and \eqref{eq:0810-23} we can prove
  \eqref{eq:0810-12}. The proof for
  \eqref{eq:0810-13} is analogous.
\end{proof}

\section{Proof of Theorem~\ref{thm:1201-1}}
\label{appendix:C}
\begin{proof}
  Since $g(x)$ is bounded and analytic in $-a \leqslant \operatorname{Im}(x) \leqslant a$, 
  from Theorem IV in the introduction of the classic work by Paley and Wiener~\cite{paley1934fourier},
  we know for $\mu,r \in \mathbb{N}, \,0 \leqslant r \leqslant \mu$,
  \begin{equation}\label{eq:1201-4}
    \begin{aligned}
    & \left\|\partial^r \mathcal{F}[u](\xi) \cdot \mathbb{I}_{\left\{|\xi|>B\right\}}\right\| \\
    = & \left\|\partial \mathcal{F}\left[u \cdot x^r\right](\xi) \cdot \mathbb{I}_{\left\{|\xi|>B\right\}}\right\| \\
    \leqslant & \left\|e^{a|\xi|} \mathcal{F}\left[u \cdot x^r\right](\xi) \cdot \mathbb{I}_{\left\{|\xi|>B\right\}}\right\| \cdot e^{-aB} \\
    \lesssim&_{\mu,u}\,e^{-aB}.
    \end{aligned}
  \end{equation}
  Hence
  \begin{equation}\label{eq:1201-5}
    \|\mathcal{F}[u](B \xi)\|_{H^\mu(\mathbb{R} \backslash[-1,1])} \lesssim_{\mu, u} e^{-a B}.
  \end{equation}
  If $\mu \notin \mathbb{N}$, by interpolation inequality \eqref{eq:1201-5} also holds.
  Combining \eqref{eq:1201-5} with Theorem~\ref{thm:bigthm} yields \eqref{eq:1201-1}. 
  This proves the error estimate result when using a set of Hermite functions for approximation.

  Morever, when $g(x)$ has singularities, from Theorem I in the introduction of~\cite{paley1934fourier},
  we know that $\mathcal{F}[u](\xi)$ decays at most exponentially, hence
  the convergence rate estimated in
  \eqref{eq:1201-1} cannot be improved to $e^{-cN^\alpha}$ for $\alpha > \frac{2}{3}$
  via Theorem~\ref{thm:bigthm}.

  From Theorem~\ref{cor:Laguerre-proj}, to analyze the error of using two sets of Laguerre functions to approximate
  $u(x)$, it is necessary to consider the decay rate of the Fourier transform of
  $u(x^2) = e^{-x^4}g(x^2)$ and its derivatives up to a certain order.
  This can be achieved by shifting the path of integration. Specifically,
  we first assume $r,\mu+\frac{1}{2} \in \mathbb{N},\, 0 \leqslant r \leqslant \mu+\frac{1}{2}$,
  consider $h(x) = x^r e^{-x^4} g(x^2)$. Since $g(x)$ is analytic and bounded in
  $-a \leqslant \operatorname{Im}(x) \leqslant a$, $g(x^2)$ is bounded and analytic in
  \begin{equation}\label{eq:def-Omega}
    \Omega = \left\{ x=t+is: -\frac{a}{2} \leqslant ts \leqslant \frac{a}{2}, \, t,s \in \mathbb{R} \right\}.
  \end{equation}
  Given that
  \begin{equation}\label{eq:Fourier-h}
    \mathcal{F}[h](\xi)=\frac{1}{\sqrt{2 \pi}} \int_{-\infty}^{+\infty} x^r e^{-x^4} g\left(x^2\right) e^{-i x \xi} d x,
  \end{equation}
  for $\xi > 0$, taking advantage of the analyticity of $h(x)$ within $\Omega$,
  we shift the path of integration to the lower boundary of $\Omega$ using
  the method of contour integration, i.e.,
  \begin{equation}\label{eq:1209-2}
    \begin{aligned}
    \mathcal{F}[h](\xi) & =\frac{1}{\sqrt{2 \pi}} \int_{-\infty}^0\left(t+i \frac{a}{2 t}\right)^r e^{-\left(t+i \frac{a}{2 t}\right)^4} g\left(\left(t+i \frac{a}{2 t}\right)^2\right) e^{-i\left(t+i \frac{a}{2 t}\right) \xi} d t, \\
    & +\frac{1}{\sqrt{2 \pi}} \int_0^{\infty}\left(t-i \frac{a}{2 t}\right)^r e^{-\left(t-i \frac{a}{2 t}\right)^4} g\left(\left(t-i \frac{a}{2 t}\right)^2\right) e^{-i\left(t-i \frac{a}{2 t}\right) \xi} d t \\
    & \triangleq \frac{1}{\sqrt{2 \pi}} I_1+\frac{1}{\sqrt{2 \pi}} I_2.
    \end{aligned}
  \end{equation}
  For $I_2$, assume $|g(x)| \leqslant G,\, \forall x \in \Omega$, we have
  \begin{equation}\label{eq:1209-3}
    \begin{aligned}
    \left|I_2\right| & \leqslant \int_0^{\infty}\left(t+\frac{a}{2 t}\right)^r e^{-\left(t^4+\left(\frac{a}{2 t}\right)^4\right)+\frac{3}{2} a^2} G e^{-\frac{a}{2 t} \xi} d t \\
    & \lesssim \int_0^{\infty} e^{-\frac{1}{2} t^4-\frac{a}{2 t} \xi} d t.
    \end{aligned}
  \end{equation} 
  Let $\varphi(t)=\frac{1}{2}t^4+\frac{a}{2 t} \xi$, then $\varphi(t)$
  has a minimum point $t_0 =\left(\frac{a \xi}{4}\right)^{\frac{1}{5}}$,
  and it is not difficult to prove that
  \begin{equation}\label{eq:1118-10}
    \varphi(t) \geqslant \varphi\left(t_0\right)+\frac{\varphi^{\prime \prime}\left(t_0\right)}{2}\left(t-t_0\right)^2 \quad \forall t>0.
  \end{equation}
  Hence
  \begin{equation}\label{eq:1118-11}
    \begin{aligned}
    \int_0^{\infty} e^{-\varphi(t)} dt & \leqslant \int_0^{\infty} e^{-\varphi\left(t_0\right)-\frac{\varphi^{\prime \prime}\left(t_0\right)}{2}\left(t-t_0\right)^2} d t \\
    & \leqslant e^{-\varphi\left(t_0\right)} \frac{1}{\sqrt{\varphi^{\prime \prime}\left(t_0\right)}} \sqrt{2 \pi}.
    \end{aligned}
  \end{equation}
  Substituting $\varphi\left(t_0\right)=\frac{5}{2}\left(\frac{a \xi}{4}\right)^{\frac{4}{5}}$,
  we obtain
  \begin{equation}\label{eq:1118-12}
    \left| I_2 \right| \lesssim \int_0^{\infty} e^{-\varphi(t)} dt \lesssim e^{-c|\xi|^{\frac{4}{5}}}.
  \end{equation}
  Similarly, $\left|I_1\right| \lesssim e^{-c|\xi|^{\frac{4}{5}}}$, combining
  with \eqref{eq:1209-2} yields
  \begin{equation}\label{eq:1118-13}
    \left|\mathcal{F}[h](\xi)\right| \lesssim_{\mu, g} e^{-c|\xi|^{\frac{4}{5}}}.
  \end{equation}
  Recall that $h(x) = x^r e^{-x^4} g(x^2)$ where $0 \leqslant r \leqslant \mu + \frac{1}{2}$,
  hence by applying Corollary~\ref{cor:Laguerre-proj} we obtain \eqref{eq:1201-2}.
  If $\mu + \frac{1}{2} \notin \mathbb{N}$, by interpolation inequality, it can be shown that \eqref{eq:1201-2} still holds.

  If $g(x^2)$ remains a bounded analytic function in the strip region
  $-b \leqslant \operatorname{Im}(x) \leqslant b$, again by Theorem IV of~\cite{paley1934fourier},
  we know for $\mu+\frac{1}{2}, r \in \mathbb{N},\,0 \leqslant r \leqslant \mu + \frac{1}{2}$,
  $h(x) = x^r e^{-x^4} g(x^2)$ satisfies
  \begin{equation}\label{eq:h-ideal-decay}
    \left|\mathcal{F}[h](\xi)\right| \lesssim_{\mu, g} e^{-c|\xi|}.
  \end{equation}
  Hence by Corollary~\ref{cor:Laguerre-proj} we get \eqref{eq:1201-3}. If
  $\mu + \frac{1}{2} \notin \mathbb{N}$, by interpolation inequality
  \eqref{eq:1201-3} also holds. This completes our proof.
\end{proof}
\end{appendices}


\bibliography{sn-bibliography}
\end{document}